\documentclass[12pt]{article}
\usepackage[utf8]{inputenc}
\usepackage{amssymb,tikz-cd,amsmath,amsthm,enumitem}
\usepackage{color}
\usepackage{tikz}
\usetikzlibrary{shapes.misc,calc,math}
\usetikzlibrary{shapes}
\usetikzlibrary{external}
\usepackage[left=0.9in,top=0.9in,right=0.9in,bottom=0.9in]{geometry}
\usepackage{mathtools}
\usepackage[hidelinks,unicode]{hyperref} % removing boxes around links
\hypersetup{bookmarksopen=true}   % Expand bookmarks in PDF viewer
\usepackage{ifthen}
\usepackage[capitalize,noabbrev]{cleveref}
\crefname{prop}{Proposition}{Propositions}   % Required to generate plural forms. See Section 8.1.3 ("Automatic \newtheorem Definitions") of cleveref documentation.
\crefname{const}{Construction}{Constructions}

\newtheorem{thm}{Theorem}[section]
\newtheorem{cor}[thm]{Corollary}
\newtheorem{prop}[thm]{Proposition}
\newtheorem{lem}[thm]{Lemma}
\newtheorem{lemma}[thm]{Lemma}
\newtheorem{claim}[thm]{Claim}

\newtheorem{quest}[thm]{Question}

\theoremstyle{definition}

\newtheorem{rem}[thm]{Remark}

\newtheorem{const}[thm]{Construction}

\newcommand{\Z}{\mathbb{Z}}

\newcommand{\del}{\overline{\delta}}
\newcommand{\gam}{\underline{\gamma}}
\newcommand{\sub}{\subseteq}
\newcommand{\Square}{J_{1/2}}
\newcommand{\Triangle}{J_{1/4}}

\newcommand{\ep}{\varepsilon}

\renewcommand{\phi}{\varphi}
\newcommand{\R}{\mathbb{R}}
\newcommand*{\deftobe}{\mathrel{\coloneqq}}   % := symbol for definitions
\newcommand{\imin}{J_{1/4}^*}   % the canonical internal minimal triangle
\DeclarePairedDelimiter{\braces}{\lbrace}{\rbrace}
\newcommand*{\setof}[1]{\braces*{#1}}   % braces for sets
\DeclarePairedDelimiter{\verts}{\lvert}{\rvert}
\newcommand*{\card}{\verts*}   % cardinality of a set
\newcommand*{\st}{\,:\,}   % "such that" for sets
\newcommand*{\limp}{\Longrightarrow}   % logical implication
\renewcommand{\subset}{\subseteq}  % I prefer ⊆ over ⊂. But I won't insist. -TM
\DeclarePairedDelimiter{\floors}{\lfloor}{\rfloor}
\newcommand*{\floor}{\floors*}
\renewcommand{\cong}{\equiv}   % Congruence modulo whatever
\newcommand{\locationofprograms}{ in the supplementary files. } 
\tikzset{
	vtx/.style={inner sep=2.1pt, outer sep=0pt, circle, fill=blue!50!white,draw=black},
	vtx2/.style={inner sep=2.1pt, outer sep=0pt, circle, fill=red!50!white,draw=black},
	vtx4/.style={inner sep=3.5pt, outer sep=0pt, circle, fill=blue!50!white,draw=black},
	triangle/.style={fill=pink,opacity=0.5,line width=1pt},
	novtx/.style={cross out, draw=red, line width=2pt},
	gsvtx/.style={inner sep=2.1pt, outer sep=0pt, rectangle, fill=green!50!white,draw=black},
	gs2vtx/.style={inner sep=1.8pt, outer sep=0pt, regular polygon,regular polygon sides=5, fill=red!50!white,draw=black},
	gs3vtx/.style={inner sep=1.4pt, outer sep=0pt, diamond, fill=yellow!50!white,draw=black},
	gs4vtx/.style={inner sep=3.5pt, outer sep=0pt, rectangle, fill=green!50!white,draw=black},
	axisedge/.style={-latex, line width=1.5pt},
}

\newcommand{\xtriangle}[3]{
	\draw[fill=green,opacity=0.5,line width=1pt] (#1)--(#2)--(#3)--(#1);
	\draw[line width=1pt] (#1)--(#2)--(#3)--(#1);
}

\newcommand{\xtriwindow}[4]{ % Make triangle window around 3 corners #1, #2, #3, adding $\leq #4$ at bottom right
	\draw[dashed, fill=yellow,opacity=0.5,line width=1pt, rounded corners] ($(#1)+(-0.4,-0.2)$)--($(#2)+(0.2,-0.2)$)--($(#3)+(0.2,0.4)$)--cycle;
	\draw ($(#2)+(0.5,-0.5)$) node[]{$\le #4 $};
}

\newcommand{\xinvtriwindow}[4]{ % Make triangle window around 3 corners #1, #2, #3, adding $\leq #4$ at bottom right
	\draw[dashed, fill=yellow,opacity=0.5,line width=1pt, rounded corners] ($(#1)+(-0.2,0.2)$)--($(#2)+(0.4,0.2)$)--($(#3)+(-0.2,-0.4)$)--cycle;
	\draw ($(#2)+(0.5,0.5)$) node[]{$\le #4 $};
}

\newcommand{\xwindow}[3]{ % Make rectangle window around #1:top-left and #2:bottom right corners, adding $\leq #3$ at bottom right%
	\begin{scope}[xshift=0cm,yshift=0cm]
		\draw[dashed, fill=yellow,opacity=0.5,line width=1pt, rounded corners] ($(#1)+(-0.2,0.2)$) rectangle ($(#2)+(0.2,-0.2)$);
		\draw ($(#2)+(0.5,-0.5)$) node[]{$\le #3$};
	\end{scope}
}

\newcommand{\xcenteredwindow}[3]{ % Make rectangle window around #1:top-left and #2:bottom right corners, adding $\leq #3$ at center%
	\begin{scope}[xshift=0cm,yshift=0cm]
		\draw[dashed, fill=yellow,opacity=0.5,line width=1pt, rounded corners] ($(#1)+(-0.2,0.2)$) rectangle ($(#2)+(0.2,-0.2)$);
		\draw ($(#2)+(-0.5,0.5)$) node[]{$\le #3$};
	\end{scope}
}

\newcommand{\xskewwindow}[4]{ % Make window around 4 corners #1, #2, #3, #4
	\draw[dashed, fill=yellow,opacity=0.5,line width=1pt, rounded corners] ($(#1)+(-0.4,-0.2)$)--($(#2)+(0.2,-0.2)$)--($(#3)+(0.4,0.2)$)--($(#4)+(-0.2,0.2)$)--cycle;
	\draw ($(#2)+(0.5,1)$) node[]{$\le 2$};
}

\newcommand{\xfivewindow}[5]{ % Make window around 5 corners
	\draw[dashed, fill=yellow,opacity=0.5,line width=1pt, rounded corners] ($(#1)+(-0.4,-0.2)$)--($(#2)+(0.2,-0.2)$)--($(#3)+(0.2,0)$)--($(#4)+(0.2,0.2)$)--($(#5)+(-0.2,0.2)$)--cycle;
	\draw ($(#2)+(0.5,1)$) node[]{$\le 2$};
}

\title{Triangle Percolation on the Grid}
\author{
	Igor Araujo\thanks{Department of Mathematics, University of Illinois at Urbana-Champaign, Urbana, IL. Email: \url{igoraa2@illinois.edu}. Research partially supported by UIUC Campus Research Board RB 22000.} 
	\qquad 
	Bryce Frederickson\thanks{Department of Mathematics, Emory University, Atlanta, GA, E-mail: \url{bfrede4@emory.edu}.}
	\qquad  
	Robert A. Krueger\thanks{Department of Mathematics, University of Illinois at Urbana-Champaign, Urbana, IL. Email: \url{rak5@illinois.edu}. Research supported the NSF Graduate Research Fellowship Program Grant No. DGE 21-4675.}
	\\
	Bernard Lidick\'y\thanks{Department of Mathematics, Iowa State University, Ames, IA. E-mail: \url{lidicky@iastate.edu}. Research of this author is supported in part by NSF grant  DMS-2152490 and Scott Hanna fellowship.}
	\qquad 
	Tyrrell B.~McAllister\thanks{Department of Mathematics and Statistics, University of Wyoming, Laramie, WY. E-mail: \url{tmcallis@uwyo.edu}.}
	\qquad  
	Florian Pfender\thanks{Department of Mathematical and Statistical Sciences, University of Colorado Denver. E-mail: \url{Florian.Pfender@ucdenver.edu}. Research is partially supported by NSF grant DMS-2152498.}
	\\ 
	Sam Spiro\thanks{Department~of Mathematics, Rutgers University, Piscataway, NJ. Email: \url{sas703@scarletmail.rutgers.edu}. Research supported by the NSF Mathematical Sciences Postdoctoral Research Fellowships Program under Grant DMS-2202730.}
	\qquad  
	Eric Nathan Stucky\thanks{E-mail: \url{stuck127@umn.edu}. Research partially supported by Czech Science Foundation grant 21-00420M, administered through Charles University, Faculty of Mathematics and Physics, Department of Algebra.}
}
\date{\today}

\begin{document}
	
	\maketitle
	
	\begin{abstract}
		We consider a geometric percolation process partially motivated by recent work of Hejda and Kala.  Specifically, we start with an initial set $X \subseteq \Z^2$, and then iteratively check whether there exists a triangle~$T \subseteq \R^2$ with its vertices in $\mathbb{Z}^2$ such that $T$ contains exactly four points of $\mathbb{Z}^2$ and exactly three points of~$X$.  In this case, we add the missing lattice point of~$T$ to~$X$, and we repeat until no such triangle exists.  We study the limit sets~$S$, the sets stable under this process, including determining their possible densities and some of their structure.
	\end{abstract}

	\newpage
	
	\section{Introduction}
	
	\textit{Bootstrap percolation} is a well-studied model in graph theory that describes how vertices in a network propagate information to their neighbors. Initiated over 40 years ago by physicists to study models of magnetization in statistical mechanics (see, e.g.,~\cite{AizenmanLebowitz1988, ChaLeiRei1979}), it has also been applied in sociology and computer science. The underlying theory has also attracted mathematical study; see, e.g.,~\cite{BalBol2003, BalPerPet2006} for early examples of the graph-theoretic formulation for this model, and the introductions of~\cite{BalBolMorSmi2023,BalBolDumMor2012} for contemporary work.
	
	More recently, bootstrap percolation has been implemented on network structures beyond graphs. In particular, the pioneering paper by Balogh, Bollob\'as, Morris and Riordin~\cite{BalBolMorRio2012} defined the following variant of the model for hypergraphs.
	Given a hypergraph $H$ with vertex set $V$ and hyperedge set $E$, for any subset $X$ of vertices, we define
	$$p_H(X) = X \cup \Big\{ v\in V: \exists e\in E \text{ such that } v\in e \text{ and } e\subset X\cup\{v\}\Big\}.$$
	In other words we grow $X$ by adding vertices that ``complete a hyperedge''. We let $p_H^n(X) = p_H( \cdots p_H(X) \cdots )$, iterated $n$ times, and we say that a subset $S$ is \textit{$H$-stable} (or simply, \textit{stable}) if $S=p_H(S)$. That original paper considered hypergraphs with vertices in a finite grid whose hyperedges are products of intervals and computed a well-studied invariant in the graph case: the smallest size of a set $X$ such that $p_H^\infty(X)=V$.
	
	This led to an active study on a class of translation-invariant hypergraphs on $\mathbb{Z}^d$, written in the language of cellular automata~\cite{GravnerGriffeath1999} and the ``update families'' of Bollob\'as, Smith, and Uzzell~\cite{BolSmiUzz2015}. Their paper spawned a great deal of interest in the update family formulation (see, e.g.,~\cite{ BalBolSmi2023, BalBolPrzSmi2016, BolDumMorSmi2023, Hartarsky2021, HartarskySzabo2023}).
	Other developments in this field have had a more geometric flavor, considering hypergraphs arising from lines in $\mathbb{Z}^d$~\cite{BalBolLeeNar2017} and finite projective planes~\cite{GerKesMesPatViz2018}. 
	
	We continue this geometric theme. The configurations that we consider are of the form $\Delta \cap \Z^{2}$ where $\Delta \subset \R^{2}$ is a lattice triangle, that is, the convex hull in~$\R^{2}$ of three non-collinear points in $\Z^2$.  Such a triangle is determined by the set of integer-lattice points that it contains. Accordingly, we abuse terminology a bit and call a subset $T \subset \Z^2$ a \textit{triangle} when~$T$ is the set of integer-lattice points in some lattice triangle $\Delta \subset \R^{2}$.  Thus, in our terminology, the triangle~$T$ is a finite set containing at least three points.  We call $T$ a \textit{minimal} triangle if $T$ contains exactly four points.  A minimal triangle is either a \textit{border} or an \textit{internal} triangle, depending on whether the fourth point (in addition to the vertices) appears on the boundary or the interior of the triangle. 
	
	In this paper, we study stable sets with respect to three hypergraphs $B, I$, and $BI$ on the vertex set $\Z^2$, whose hyperedges are all of the border, internal, and minimal triangles, respectively. The triangles that define this percolation are simple structures, but they allow for long-distance effects to occur in just one step of the process, because the triangles may be extremely thin (see \cref{fig:I5_I7}). 
	
	Our results are framed in a somewhat different language because our original motivation for this study came instead from recent developments in the theory of  universal quadratic forms. While attempting to improve certain dimension bounds in \cite{BlomerKala2015}, Hejda and Kala \cite{HejdaKala2020} obtained a remarkable result about the totally positive integers $\mathcal{O}_K^+$ for (totally positive) quadratic number fields $K$: the abstract isomorphism type of the semigroup $\mathcal{O}_K^+$ already suffices to determine the quadratic extension $K$. The argument in their proof can be reframed as follows: a canonical choice of generators $(\dotsc, x_{-1}, x_0, x_1, \dotsc)$ comes with relations of the form $x_{i-1} + x_{i+1} = c_i x_i$ for all $i$. Once any two consecutive points are fixed, successively applying these relations fixes all the other points, which determines all relations among them. In this sense, all relations among these generators are determined by ``local'' ones. 
	
	For number fields of higher rank, the analogous generators $\mathcal{O}_K^+$ are much more elusive, and not much is known about the relations (see, e.g.,~\cite{Brunotte1983, KalaTinkova2022, KrasTinkZem2020}). Nonetheless, it is natural to wonder if there is an abstract Hejda--Kala type argument showing that only some initial data and the ``local'' relations could determine the entire semigroup structure of $\mathcal{O}_K^+$. However, note that at each stage it must be possible to determine at least one additional point from those which have already been determined; in $\mathbb{R}^1$ (rank 2) this is certainly possible, but in higher dimensions the situation is less clear. Our stable sets may be understood roughly as a worst-case analysis of such combinatorial obstructions in rank 3.
	
	\subsection{Definitions and results}\label{subsec:DefinitionsAndResults}
	
	As mentioned before, we call a subset $T \subset \Z^2$ a \textit{triangle} if $T$ is the set of integer-lattice points in the convex hull $\Delta$ in $\R^2$ of three noncollinear points in $\Z^2$. Such a triangle $T$ uniquely determines $\Delta$ (which is a triangle in the conventional sense) so we may speak of the area, vertices, boundary points, and interior points of $T$. The triangle $T$ is \textit{unimodular} if the vertices are an affine basis of the lattice~$\Z^{2}$ over $\Z$. Equivalently, the unimodular triangles are the triangles that have area $1/2$.  Pick's formula $A = i + b/2 - 1$ expresses the area $A$ of $T$ in terms of the number~$b$ of lattice points on its boundary and the number $i$ of lattice points in its interior.  From this relation, unimodular triangles are easily recognized as precisely the triangles that contain exactly $3$ points. Recall also that a  \textit{minimal} triangle is a triangle that contains exactly $4$ points, and it is either a \textit{border} or an \textit{internal} triangle, depending on whether the non-vertex point appears on the triangle's boundary or the interior.  From Pick's formula, the border triangles are precisely the triangles of area $1$, and every internal triangle has area $3/2$ (though some triangles of area $3/2$ are not internal triangles).
	
	A unimodular transformation of $\Z^{2}$ is an affine-linear bijection from
	$\Z^2$ to $\Z^2$.  Equivalently, the unimodular transformations
	of $\Z^{2}$ are the maps of the form $x \mapsto Mx + b$, where $M$ is a $2
	\times 2$ matrix with integer entries and determinant $\pm 1$, and
	$b \in \Z^{2}$.  Note that unimodular transformations also
	preserve the number of lattice points on a line segment.  Thus unimodular
	transformations preserve the property of a triangle being
	unimodular, border, or internal.
	
	Let $S$ be a subset of $\Z^{2}$.  We say that $S$ is
	\textit{B-stable} if no border triangle has exactly three of its
	points in $S$.  We call $S$ \textit{I-stable} if no internal
	triangle has exactly three of its points in $S$.  Finally, we call
	$S$ \textit{BI-stable} if $S$ is both B-stable and I-stable.  Since
	$\Z^2$ itself is BI-stable, we say that a proper subset $S$ of $\Z^2$ is a \textit{maximal}
	B/I/BI-stable set if $S$ is B/I/BI-stable and no B/I/BI-stable
	proper subset of $\mathbb Z^2$ properly contains $S$.
	
	\begin{figure}[h!]
		\centering
		\begin{tikzpicture}[scale=0.8]
			\draw (0,0)grid(7,4);
			\begin{scope}[xshift=1cm,yshift=1cm]
				\draw 
				(1,0)coordinate (a)
				(0,1)coordinate (b)
				(2,2)coordinate (c)
				;
				\xtriangle{a}{b}{c}
				\foreach \x in {a,b,c}{
					\draw (\x)node[vtx]{};
				}
				\draw(1,1) node[novtx]{};
			\end{scope}
			
			\begin{scope}[xshift=4cm,yshift=1cm]
				\draw 
				(1,0)coordinate (a)
				(0,1)coordinate (b)
				(2,2)coordinate (c)
				;
				\xtriangle{a}{b}{c}
				\foreach \x in {a,b}{
					\draw (\x)node[vtx]{};
				}
				\draw(1,1) node[vtx]{};
				\draw(c) node[novtx]{};
			\end{scope}
			\draw (3.5,0) node[below]{Not I-stable};
		\end{tikzpicture}
		\hskip 2em
		\begin{tikzpicture}[scale=0.8]
			\draw (0,0)grid(10,4);
			\begin{scope}[xshift=1cm,yshift=1cm]
				\draw 
				(1,0)coordinate (a)
				(0,1)coordinate (b)
				(2,1)coordinate (c)
				;
				\xtriangle{a}{b}{c}
				\foreach \x in {a,b,c}{
					\draw (\x)node[vtx]{};
				}
				\draw(1,1) node[novtx]{};
			\end{scope}
			
			\begin{scope}[xshift=4cm,yshift=1cm]
				\draw 
				(1,0)coordinate (a)
				(0,1)coordinate (b)
				(2,1)coordinate (c)
				;
				\xtriangle{a}{b}{c}
				\foreach \x in {a,b}{
					\draw (\x)node[vtx]{};
				}
				\draw(1,1) node[vtx]{};
				\draw(c) node[novtx]{};
			\end{scope}
			
			\begin{scope}[xshift=7cm,yshift=1cm]
				\draw 
				(1,0)coordinate (a)
				(0,1)coordinate (b)
				(2,1)coordinate (c)
				;
				\xtriangle{a}{b}{c}
				\foreach \x in {b,c}{
					\draw (\x)node[vtx]{};
				}
				\draw(1,1) node[vtx]{};
				\draw(a) node[novtx]{};
			\end{scope}
			\draw (5,0) node[below]{Not B-stable};
		\end{tikzpicture}
		\caption{Configurations that are not stable.  Specifically, if a set contains only the three blue points in one of the minimal triangles but not the red cross in that triangle, then that triangle demonstrates that the set is not stable in the indicated sense.
		}
		\label{fig:stable}
	\end{figure}
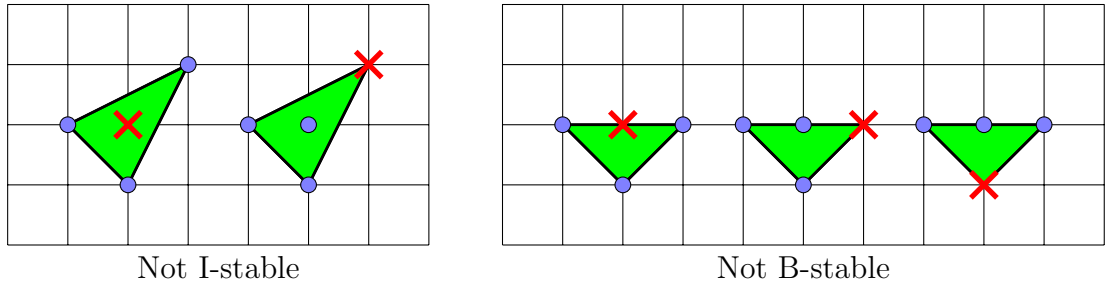		
	
	Roughly speaking, our results fall into two categories: results
	regarding the possible densities of
	B/I/BI-stable sets and results regarding the existence and uniqueness
	of B/I/BI-stable subsets that satisfy various
	conditions, such as maximality or the presence of particular
	point-configurations.
	
	Our first two theorems summarize our results on the densities of
	stable sets. We define the \textit{\textup{(}upper\textup{)} density} of a subset $X \subset \Z^2$ to be
	\begin{equation*}
		\del(X)
		\deftobe
		\limsup_{n\to \infty}
		\frac{\card{X \cap [-n,n]^2}}{\card{[-n,n]^2}},
	\end{equation*}
	where, here and throughout, $[a,b]$ denotes the interval of
	integers $\setof{i \in \Z \st a \le i \le b}$. While $\del(X)$ is
	a fairly natural way to define the density of a subset $X \sub
	\Z^2$, it is not always true that $\del(\phi(X))=\del(X)$ for
	every unimodular transformation $\phi$. 
	As an illustration, the set $\{(x,y) \in \Z^2: xy >0 \}$ is mapped to $\{(x,y) \in \Z^2: x(y-x) >0 \}$ under the unimodular map $(x,y) \mapsto (x,x+y)$, and while the former has density~$1/2$, the latter has density~$1/4$.
	However, density is preserved under unimodular transformation for sets $X$ that are periodic in two independent directions, which will
	be the case for most of the subsets $X \sub \Z^2$ that we
	consider. Our first theorem gives the possible values of these
	densities for stable sets (periodic or aperiodic).
	\begin{thm}\label{thm:main}
		Let $S\subsetneq \Z^2$.
		\begin{enumerate}[label=(\alph*)]
			\item \label{border}
			If $S$ is B-stable, then $\del(S)\le 1/4$.  Moreover, for
			every $\delta \le 1/4$, there exists a BI-stable set~$S$ with
			$\del(S) = \delta$; for instance, every subset of $2\Z^2 \coloneqq 2 \Z \times 2 \Z$ is
			BI-stable.
			
			\item \label{internal} 
			If $S$ is I-stable, then $\del(S)\le 1/2$.  Moreover, for
			every $\delta \le 1/2$, there exists an I-stable set~$S$ with
			$\del(S) = \delta$; for instance, every subset of $\Z \times
			2\Z$ is I-stable.
		\end{enumerate}
	\end{thm}
	
	The two upper bounds in \cref{thm:main} are proved in
	\cref{prop:main-border,prop:main-internal}, respectively.
	The BI-stability of $2\Z^2$ and its subsets is
	\cref{cor:completely-stable-densities}, and the I-stability of
	$\Z\times 2\Z$ and its subsets is
	\cref{cor:internal-stable-densities}.
	
	For B-stable sets (and in particular
	for BI-stable sets), the upper bound in \cref{thm:main} is
	attained by $2\Z^{2}$. Our next theorem shows that any stable set that is ``structurally far''
	from $2\Z^2$ must have a significantly
	lower density.  That is, the maximally dense construction for B-stable
	sets is ``stable'' in the sense commonly used
	in extremal combinatorics. In order to state this result
	precisely, we define the \textit{\textup{(}lower\textup{)} consecutive density}
	$\gam(X)$ of a set $X \subseteq \Z^2$ as follows: Let $\Gamma(X)$
	denote the set of points $(x,y) \in X$ such that $(x+1,y) \in X$,
	and let
	\begin{equation*}
		\gam(X)
		\deftobe
		\liminf_{n\to\infty}
		\frac{\card{\Gamma(X) \cap [-n,n]^2}}{\card{[-n,n]^2}}.
	\end{equation*}
	Note that $\gam(2\Z^2)=0$. Thus, the consecutive density provides one measure of the
	``distance'' between a set $X\subset \Z^{2}$ and~$2\Z^{2}$. We
	show that, for B-stable proper subsets \mbox{$S \subsetneq \Z^{2}$}, if $\gam(S)$ is
	large, then $\del(S)$ must be correspondingly smaller than the
	maximum given by \cref{thm:main}.  The following is proved in
	\cref{sec:consecutive}.
	\begin{thm}\label{thm:stab}
		If $S\subsetneq \Z^2$ is B-stable (in particular, if $S$ is
		BI-stable), then $\gam(S) \leq 1/9$ and $\del(S) \le (1 -
		\gam(S))/4$.  Moreover, there is a BI-stable sets $S$ with $\gam(S) = 1/9$ and $\del(S)=2/9$, and a BI-stable set $S'$ with $\gam(S')=0$ and $\del(S')=1/4$. 
	\end{thm}
	
	We remark that no analogous result is possible for I-stable sets. For example, there exist two very different I-stable sets $I_2$ and $\Square$ that both have the maximum possible density $1/2$. (See \cref{const:columns,const:J14_J12} below for the definitions of $I_2$ and $\Square$, respectively.)
	
	Our next two theorems summarize our results regarding the
	existence and uniqueness of stable sets under various conditions.
	The next theorem shows that, if a stable set $S$ contains certain
	small configurations, then $S$ must be one of a very few different
	sets. For the definitions of the sets $\Triangle$ and $\Square$ referred to in the theorem, see
	\cref{const:J14_J12} and \cref{fig:J14_J12} below.
	
	\begin{thm}\label{thm:StableSetsContainingConfigurations}
		Let $S \subset \Z^{2}$.
		{\makeatletter % prevent a pagebreak before the list.
			\@beginparpenalty=10000
			\makeatother
			\begin{enumerate}[label=(\alph*)]
				\item 
				\label{thmpart:StableSetsContainingConfigurations-BStable}
				If $S$ is B-stable and contains 
				any three points of a border triangle, then $S = \Z^{2}$. 
				
				\item  
				\label{thmpart:StableSetsContainingConfigurations-BIStable}
				If $S$ is BI-stable and contains any three points of a minimal triangle, then $S = \Z^{2}$.
				
				\item
				\label{thmpart:StableSetsContainingConfigurations-IStable}
				If $S$ is I-stable and contains any three points of an internal triangle, then, up to a unimodular transformation, $S$ is one of $\Triangle$,
				$\Square$, or $\Z^{2}$.
		\end{enumerate}}
	\end{thm}
	
	\begin{rem}
		In the language of bootstrap percolation, this theorem implies that our models are ``supercritical'' in the classification of Bollob\'as, Smith, and Uzzell \cite{BolSmiUzz2015}. Indeed, they admit finite \emph{percolating sets} (i.e.\ finite $X$ with $p_H^\infty(X)=\Z^2$); these have minimum size 3 for the hypergraphs B and BI, and 5 for I.
	\end{rem}
	
	Parts~\ref{thmpart:StableSetsContainingConfigurations-BStable} and~\ref{thmpart:StableSetsContainingConfigurations-BIStable} of
	\cref{thm:StableSetsContainingConfigurations} are proved in
	\cref{prop:unimodularTriangle}, while
	Part~\ref{thmpart:StableSetsContainingConfigurations-IStable}
	follows from \cref{cor:X2IntExtremal}. Note that, in particular, if $S$ contains a unimodular triangle, then $S$ meets all of the ``contains any three points of \ldots'' conditions in \cref{thm:StableSetsContainingConfigurations}.
	
	Our final theorem summarizes our results regarding maximal stable
	sets.  The first two parts of this theorem show that the study of
	stable sets reduces to the study of maximal stable sets.
	
	\begin{thm}\label{thm:MaximalStableSetsAndSubsets}\mbox{}
		\begin{enumerate}[label=(\alph*)]
			\item 
			\label{thm:MaximalStableSetsAndSubsets-BBISubsets}
			If $S$ is a B/BI-stable proper subset of $\Z^{2}$,
			then every subset of $S$ is B/BI-stable (of the same stability type).
			
			\item 
			\label{thm:MaximalStableSetsAndSubsets-ISubsets}
			If $S$ is an I-stable subset of $\Z^{2}$ that is not, up to
			unimodular transformation, either $\Triangle$, $\Square$, or
			$\Z^{2}$, then every subset of $S$ is I-stable.
			
			\item  
			\label{thm:MaximalStableSetsAndSubsets-Maximal}
			Every B/I/BI-stable proper subset of $\Z^{2}$ is contained in a maximal B/I/BI-stable set (of the same stability type).
			
			\item  
			\label{thm:MaximalStableSetsAndSubsets-Nonperiodic}
			There exist non-periodic maximal B/I/BI-stable sets.
		\end{enumerate}
	\end{thm}
	
	Part~\ref{thm:MaximalStableSetsAndSubsets-BBISubsets} of
	\cref{thm:MaximalStableSetsAndSubsets} is proved in
	\cref{cor:BorderStableSubsets}.
	Part~\ref{thm:MaximalStableSetsAndSubsets-ISubsets} follows
	from \cref{cor:structure-internal}.
	Part~\ref{thm:MaximalStableSetsAndSubsets-Maximal} is proved in
	\cref{prop:zorns}, and
	Part~\ref{thm:MaximalStableSetsAndSubsets-Nonperiodic} is proved
	in
	\cref{prop:Iaperiodic,prop:aperodic}.
	
	This paper is organized as follows.  In
	\cref{sec:BIStableSets,sec:BStableSets,,sec:IStableSets} we
	consider BI-stable, B-stable, and I-stable sets respectively.  In
	\cref{sec:maximal} we study maximal stable sets.  We conclude with
	some open problems in \cref{sec:ConcludingRemarks}.
	
	\section{BI-stable Sets}
	\label{sec:BIStableSets}
	
	In this section we give several constructions of BI-stable sets.
	Here and throughout, we use the standard notations 
	\begin{align*}
		nX 
		&\deftobe
		\setof{nx \st x \in X},
		&
		X + Y
		&\deftobe
		\setof{x + y \st \text{$x \in X$ and $y \in Y$}},
		&
		X + z
		&\deftobe
		X + \setof{z},
	\end{align*}
	for integers $n \in \Z$, subsets $X, Y \subset \Z^2$, and points $z
	\in \Z^{2}$.  We will also write $[n]$ for the interval $[1, n] = \setof{1, \dotsc, n}$ of integers.
	
	Our first construction is the set $n\Z^2$; see \cref{fig:stable14} for the case $n=2$.
	
	\begin{figure}[h!]
		\centering
		\begin{tikzpicture}[scale=0.36]
			\draw[-latex,line width=2pt] (-0.5,6) -- (16,6); 
			\draw[-latex,line width=2pt] (6,-0.5) -- (6,16); 
			\clip (-0.5,-0.5) rectangle (15.5,15.5);
			\draw (-1,-1)grid(16,16);
			\begin{scope}[xshift=0cm,yshift=0cm]
				\foreach \x in {0,...,8}{
					\foreach \y in {0,...,8}{
						\draw (2*\x,2*\y)node[vtx]{};
					}
				}
			\end{scope}
		\end{tikzpicture}
		\caption{The blue points denote $2\Z^2$, a BI-stable set of density $1/4$.}
		\label{fig:stable14}
	\end{figure}
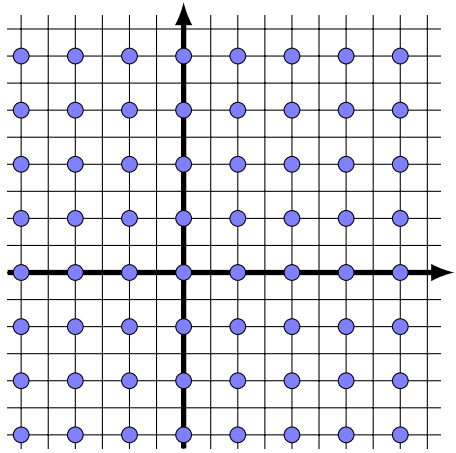
	
	\begin{prop}\label{prop:nZConstruction}
		Let $n \ge 2$. Every minimal triangle intersects $n \Z^2$ in at most $2$ points.  In particular, every subset of $n\Z^2$ is BI-stable.
	\end{prop}
	\begin{proof}
		The line segment between any two points of $n\Z^2$ contains at least $n+1 \ge 3$ lattice points including the endpoints, so any triangle containing at least three points of $n\Z^2$ will contain at least $2n+1>4$  lattice points and thus cannot be a border or internal triangle.
	\end{proof}
	
	With this lemma we can prove the existence statement of
	\cref{thm:main}\ref{border} for BI-stable sets with small density.
	
	\begin{cor}\label{cor:completely-stable-densities}
		For every $\delta\le 1/4$, there exists a set $S\subseteq 2\Z^2$ with $\del(S)=\delta$ that is BI-stable (and hence also B-stable and I-stable).
	\end{cor}
	\begin{proof}
		Let $X^{(0)} \deftobe \emptyset$, and iteratively define $X^{(n)}\sub 2\Z^2 \cap [-n,n]^2$ to be any set containing $X^{(n-1)}$ and which has exactly $\min\setof{\floor{\delta (2n+1)^2}, \card{2\Z^2 \cap [-n,n]^2}}$ points of $2\Z^2$.  Let $S \deftobe \bigcup_{i \ge 0} X^{(i)}$.  By construction we have $\del(S)=\delta$. Since $S$ is a subset of $2\Z^2$, $S$ is BI-stable by \cref{prop:nZConstruction}. 
	\end{proof}
	
	We know of no other maximal BI-stable sets with upper density
	$1/4$ besides $2\Z^2$ and its unimodular transforms; see
	\cref{quest:borderUnique} for more on this.  In fact,
	\cref{thm:stab} shows that any BI-stable set that is ``far'' from
	$2\Z^2$ (in a certain precise sense) has density strictly smaller
	than~$1/4$.  The following gives an example of a BI-stable set
	that is ``far'' from $2\Z^2$ and which still has a relatively
	large density of $2/9$.  See \cref{fig:2/9} for an illustration.
	
	\begin{const}\label{const:2/9}
		Let $S_{2/9}^* \deftobe \{(0,0), (1,0)\}$ and $S_{2/9} = S_{2/9}^*
		+ 3\Z^2$.
	\end{const}
	
	\begin{figure}[h!]
		\centering
		\begin{tikzpicture}[scale=0.36]
			\draw[-latex,line width=2pt] (-0.5,6) -- (15,6); 
			\draw[-latex,line width=2pt] (6,-0.5) -- (6,15); 
			\clip (-0.5,-0.5) rectangle (14.5,14.5);
			\draw (-1,-1)grid(16,16);
			\begin{scope}[xshift=0cm,yshift=0cm]
				\foreach \x in {0,...,4}{
					\foreach \y in {0,...,4}{
						\draw (3*\x,3*\y)node[vtx]{};
						\draw (3*\x+1,3*\y)node[vtx]{};
					}
				}
			\end{scope}
		\end{tikzpicture}
		\caption{The blue points are from $S_{2/9}$, a BI-stable set of density $2/9$ defined in \cref{const:2/9}.}
		\label{fig:2/9}
	\end{figure}
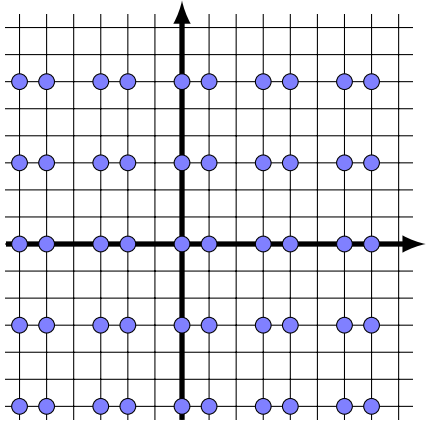
	
	\begin{prop} \label{prop:S_2/9}
		The set $S_{2/9}$ is BI-stable with $\del(S_{2/9})=2/9$.
	\end{prop}
	\begin{proof}
		First, recall that the area of a triangle with vertices $(a_1,a_2)$, $(b_1,b_2)$, and $(c_1,c_2)$ is
		$$\frac{1}{2}\det
		\begin{bmatrix}
			b_1-a_1 & b_2-a_2 \\
			c_1-a_1 & c_2-a_2
		\end{bmatrix}.$$
		When these vertices are in $S_{2/9}$, the second column of this matrix is a multiple of 3, and thus the area of any such triangle is a multiple of $3/2$.
		
		Thus $S_{2/9}$ contains neither any unimodular triangle nor the vertices of any border triangle. Therefore, it intersects any border triangle in at most $2$ points and so is B-stable.
		
		Moreover, by the pigeonhole principle, among any three points in $S_{2/9}$ there must be at least 
		two points that are in the same equivalence class modulo $3\Z^2$. These two points have at least two additional lattice points on the line segment joining them. Therefore, $S_{2/9}$ does not contain the vertices of any internal triangle. As noted before, $S_{2/9}$ also does not contain a unimodular triangle, and thus $S_{2/9}$ intersects any internal triangle in at most $2$ points. 
		Therefore, $S_{2/9}$ is also I-stable.
	\end{proof}
	
	The proofs in this section showed that the relevant sets $S$ are
	BI-stable by using a stronger property, namely that $S$ contains
	at most two points of any minimal triangle.  This is not a
	coincidence; as we shall soon see
	(\cref{prop:unimodularTriangle}\ref{lempart:BIstableIFF} below),
	this property is equivalent to BI-stability for proper subsets of
	the integer lattice.

	\section{B-stable Sets}
	\label{sec:BStableSets}
	
	In this section we establish our bounds on B-stable sets, namely that $\del(S)\le 1/4$ whenever $S\subsetneq \Z^2$ is a B-stable set. Our most important structural result towards this end is the following.
	\begin{prop}\label{prop:unimodularTriangle}
		Let $S\subsetneq \Z^2$.  Then $S$ is:
		\begin{enumerate}[label=(\alph*)]
			\item B-stable if and only if $S$ contains at most $2$ points from any border triangle, and
			\item \label{lempart:BIstableIFF} BI-stable if and only if $S$ contains at most $2$ points from any minimal triangle. 
		\end{enumerate}
	\end{prop}
	\begin{proof}
		We begin with part (a). Evidently, if $S$ contains at most $2$ points from any border triangle, then by definition $S$ is B-stable.  Conversely, assume that $S$ is B-stable but contains at least $3$ points from a border triangle,  which implies $S$ contains all $4$ points of this border triangle. Applying a unimodular transformation to $S$, we may assume without loss of generality that this triangle is $\{(0,0),(1,0),(2,0),(0,1)\}$. Observe that, for each $t\in\Z$, each of the border triangles $\{(0,0),(1,0),(2,0),(t, 1)\}$ intersects $S$ in at least three points, and so the B-stability of $S$ implies that $S$ contains every point of the form $(t, 1)$.  Thus, $S$ intersects each of the triangles $\{(0,1),(1, 1),(2, 1),(t, 2)\}$ in at least three points, and hence contains every point of the form $(t,2)$. Iterating this procedure, both up and down $\Z^2$, we conclude that $S = \Z^2$, completing the proof of part (a).
		
		We will show that part (b) follows from part (a). As before, containing at most two points from any minimal triangle immediately implies BI-stability. Conversely, if $S \subsetneq \Z^2$ is BI-stable, then $S$ contains at most $2$ points from any border triangle by part (a); if $S$ contains at least $3$ points from an internal triangle, then $S$ contains that internal triangle. Applying a unimodular transformation, we may assume without loss of generality that this triangle has points $\{(0,0),(1,0),(-1,-1),(0,1)\}$. But then this triangle contains $3$ points $(0,0), (1,0), (0,1)$ from a border triangle, a contradiction.
	\end{proof}
	
	\begin{cor}
		\label{cor:BorderStableSubsets}
		If $S \subsetneq \Z^2$ is a B-stable set, then any subset of $S$ is also B-stable. Similarly, if $S \subsetneq \Z^2$ is a BI-stable set, then any subset of $S$ is also BI-stable.
	\end{cor}
	\Cref{cor:BorderStableSubsets} follows immediately from \cref{prop:unimodularTriangle} and yields a large supply of examples.  We note that the situation for I-stable sets is somewhat more complicated (see \cref{cor:X2IntExtremal}).
	
	With \cref{prop:unimodularTriangle} we can also quickly derive the following.
	
	\begin{lem}
		\label{lem:weak-border}
		If $S \subsetneq \Z^2$ is B-stable, then $\card{S \cap ([3] \times [2])} \le 2$.
	\end{lem}
	\begin{proof}
		Let $S \subsetneq \Z^2$ be B-stable. If $S$ contains at least 
		three points of $[3] \times [2]$, then without loss of 
		generality we can assume it contains at least two points from the 
		set $X \deftobe \{(1,1),(2,1),(3,1)\}$. If $S$ contains $X$, then 
		$S$ contains three points of a border triangle, contradicting \cref{prop:unimodularTriangle}. Otherwise, $S$ contains two points of $X$ and some point of the form $(x,2)$, and so again $X$ contains three points of a border triangle.
	\end{proof}
	From this result, it follows that $\del(S) \le 1/3$ for any B-stable set $S \subsetneq \Z^2$.  To see this, note first that any translate of $S$ is also B-stable, and therefore $S$ contains at most two points of any translate of $[3]\times [2]$. Now cover $[-n,n]^2$ by $\lceil\frac{2n+1}{3}\rceil^2$ translates of $[3]\times[2]$. By the previous lemma we know that $S$ contains at most a third of the points from each of these translates, giving the bound. 
	
	We prove the optimal bound $\del(S) \le 1/4$ similarly, using the following more subtle replacement of \cref{lem:weak-border}.
	\begin{lem}\label{lem:6x6}
		If $S \subsetneq \Z^2$ is B-stable, then $\card{S \cap [6]^2} \le 9$. 
	\end{lem}
	We provide two proofs of this result. The first uses a computer search, the code for which may be found\locationofprograms
	For this we simply exhaustively checked that all B-stable subsets of $[6]^2$ are either $[6]^2$ or have at most $9$ points.
	However, we also give a human-readable proof, which can be found in the appendix.
	
	We now prove our main result for this section, which together with \cref{cor:completely-stable-densities} completes the proof of \cref{thm:main}\ref{border}.
	\begin{prop}
		\label{prop:main-border}
		If $S \subsetneq \Z^2$ is B-stable, then $\del(S)\le 1/4$.
	\end{prop}
	\begin{proof}
		
		Cover $[-n,n]^2$ by $\lceil\frac{2n+1}{6}\rceil^2$ pairwise disjoint translates of $[6]\times[6]$. By \cref{lem:6x6}, each of these translates contains at most $9$ points of $S$, for a total of at most 
		\[
		9\left\lceil\frac{2n+1}{6}\right\rceil^2\le
		9\left(\frac{2n+1}{6}+\frac56\right)^2
		=\frac14((2n+1)^2+20n+35)
		\]
		points in $S \cap [-n,n]^2$.
	\end{proof}
	
	\begin{rem} \label{rem:6x6}
		The $[6] \times [6]$ translates used in the proof of \cref{prop:main-border} are the smallest translates that are able to prove the theorem this way. Indeed, for every $a < 6$ and every $b$, at least one of the B-stable sets $2\Z^2$ or $S_{2/9}$ from \cref{const:2/9} has more than $1/4$ of the points in $[0,a-1] \times [0,b]$.
	\end{rem}

	\subsection{Consecutive pairs in B-stable sets}\label{sec:consecutive}
	
	Here we prove \cref{thm:stab}, which roughly says that if a B-stable set has many consecutive pairs, then its density cannot be close to the maximum of $1/4$. For this we need two lemmas concerning $\Gamma(S)$, which we recall is the set of points $(x,y) \in S$ such that $(x+1,y) \in S$.
	
	\begin{lem}\label{lem:3x3}
		Let $S \subsetneq \Z^2$ be a B-stable set. Then $\card{\Gamma(S) \cap [3]^2} \leq 1$. 
	\end{lem}
	\begin{proof}
		Suppose $(x,y) \in \Gamma(S) \cap [3]^2$, which means $(x,y), (x+1,y) \in S$.
		Checking all the cases, we easily see that any additional point 
		$(x',y')$ with $|x-x'|\le 2$ and $|y-y'|\le 2$ other than 
		$(x-2,y)$ is in a border triangle with $(x,y)$ and $(x+1,y)$, so 
		$(x',y')\notin S$ by \cref{prop:unimodularTriangle}, and thus $(x',y')\notin\Gamma(S)$. Finally, $(x-2,y)\notin\Gamma(S)$ since $(x-1,y)\notin S$. 
	\end{proof}
	
	We obtained the next result by computer search, similar to \cref{lem:6x6}. The code we used may be found\locationofprograms
	
	\begin{lem}\label{lem:consec16}
		Let $S \subsetneq \Z^2$ be a B-stable set. If $\Gamma(S) \cap ([6]\times[12]) \neq \emptyset$, then $$\card{S \cap ([6]\times[12])} \leq 16.$$
	\end{lem}
	
	We note that \cref{lem:6x6} implies that for any $S \subsetneq \Z^2$ which is B-stable, $|S \cap ([6] \times [12])| \leq 18$, so \cref{lem:6x6} says we have a stronger bound provided $\Gamma(S) \cap ([6]\times[12]) \neq \emptyset$. We can now prove \cref{thm:stab}.
	
	\begin{proof}[Proof of \cref{thm:stab}]
		Let $S \subsetneq \Z^2$ be B-stable. To prove that $\gam(S) \leq 1/9$, we use a similar argument as in the proof of \cref{prop:main-border}. Cover $[-n,n]^2$ by $\left\lceil\frac{2n+1}3\right\rceil^2$ translates of $[3]^2$, each of which contains at most one point of $\Gamma(S)$ by \cref{lem:3x3}. This implies the desired result.
		
		Now we prove that $\del(S) \leq \frac{1}{4}(1-\gam(S))$. Let $\ep>0$, and choose $n$ large enough such that
		\begin{equation}\label{eq:6x12}
			\card{\Gamma(S) \cap [-n,n]^2} \geq (\gam(S)-\ep) (2n+1)^2\geq (\gam(S)-\ep) 4n^2 .
		\end{equation}
		Call a translate of $[6]\times [12]$ {\em deficient} if it contains a point in $\Gamma(S)$. By \cref{lem:consec16}, such a deficient translate contains at most $16$ points in $S$. Since $[6] \times [12]$ can be covered by $8$ translates of $[3]^2$, we have by \cref{lem:3x3} that every translate of $[6] \times [12]$ contains at most $8$ points in $\Gamma(S)$.
		
		Cover $[-n,n]^2$ with $2\left\lceil\frac{2n+1}{12}\right\rceil^2$ translates of $[6]\times [12]$. Using~\eqref{eq:6x12}, this implies that there are at least $\frac12 (\gam(S)-\ep)n^2$ deficient translates among these. Therefore,
		\begin{eqnarray*}
			\card{S \cap [-n,n]^2} &\leq& 18\times 2 \left\lceil\frac{2n+1}{12}\right\rceil^2- 2\times\frac12(\gam(S)-\ep)n^2\\
			&\le&
			(1-\gam(S)+\ep)n^2+12n+36\\
			&<& \frac14(1-\gam(S)+\ep)(2n+1)^2+12n+36.
		\end{eqnarray*}
		Since $\ep>0$ was chosen arbitrarily, this shows that $\del(S) \le (1 - \gam(S))\frac{1}{4}$.
		
		The bound $\gam(S) \leq 1/9$ is best possible, as exhibited by $S_{2/9}$ from \cref{const:2/9}.  Similarly the bounds on $\del(S)$ are best possible when $\gam(S)\in \{0,1/9\}$ by considering $2\Z^2$ and $S_{2/9}$, respectively.
	\end{proof}
	
	\begin{rem}\label{rem:6x12}
		As with \cref{rem:6x6}, the $[6] \times [12]$ translates used in the proof of \cref{thm:stab} are the smallest translates which work for this purpose. Indeed, by \cref{rem:6x6} such an $[a] \times [b]$ window should have $a$ and $b$ divisible by $6$, but there are B-stable sets $S \subset [6] \times [6]$ with $\Gamma(S) \neq \emptyset$ and $\card{S}=9$, as shown in \cref{fig:9pts}.
	\end{rem}
	
	\begin{rem}
		We do not expect the bound $\del(X) \leq (1-\gam(X)) \frac{1}{4}$ for B-stable $X \subsetneq \Z^2$ to be sharp for all values of $\gam(X)$.  In fact, we can not rule out that having even a single consecutive pair in $X$ decreases its density by a fixed constant; see \cref{quest:conscutive}.
	\end{rem}

	\subsection{B-stability without BI-stability}
	Most of the constructions of B-stable sets that we know of are also BI-stable.  However, the following gives an example of a set of positive density which is B-stable but not BI-stable.
	\begin{const}\label{con:B1/12}
		Let $B_{1/12}^* \deftobe \{(1,0),(0,1),(-1,-1)\}$ and $B_{1/12} \deftobe B_{1/12}^*+6\Z^2$; see \cref{fig:B2/12}.
	\end{const}
	
	\begin{figure}
		\centering
		\begin{tikzpicture}[scale=0.36]
			\draw[axisedge] (-1.5,6) -- (15,6); 
			\draw[axisedge] (6,-1.5) -- (6,15); 
			\clip (-1.5,-1.5) rectangle (14.5,14.5);
			\draw (-2,-2)grid(15,15);
			\begin{scope}[xshift=0cm,yshift=0cm]
				\foreach \x in {0,...,3}{
					\foreach \y in {0,...,3}{
						\begin{scope}[xshift=6*\x cm,yshift=6*\y cm]
							\draw (1,0)node[vtx]{};
							\draw (0,1)node[vtx]{};
							\draw (-1,-1)node[vtx]{};
						\end{scope}
				}}
			\end{scope}
		\end{tikzpicture}    
		\caption{
			The blue points are from $B_{1/12}$, a B-stable set of density $1/12$ that is not I-stable. See \cref{con:B1/12}.}
		\label{fig:B2/12}
	\end{figure}
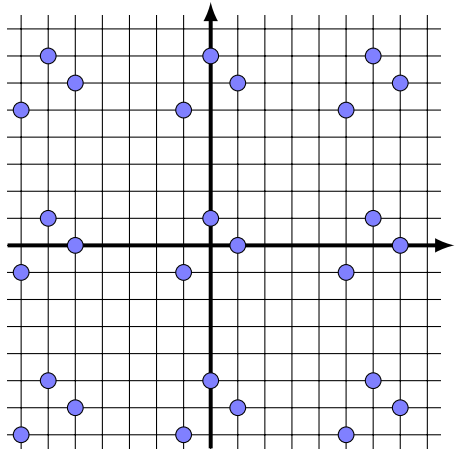

	\begin{prop}
		The set $B_{1/12}$ is B-stable, but not I-stable, with $\del(B_{1/12})=1/12$.
	\end{prop}
	\begin{proof}
		First note that $B_{1/12}$ is not I-stable, since $B_{1/12}^*$ are the three vertices of an internal triangle, but $B_{1/12}$ does not contain the fourth point, $(0,0)$. It is easy to check that $\del(B_{1/12}) = 3/36 = 1/12$.
		
		To check that $B_{1/12}$ is B-stable, we let $x$, $y$, and $z$ be three points in $B_{1/12}$ and prove that they are not contained in the same border triangle.  To do this, it suffices to confirm they are not all consecutive lattice points in any line, and also that they do not form a triangle of area~$1$ or~$1/2$.
		
		By way of contradiction, suppose first that $x$, $y$ and $z$ are collinear and consecutive.  In particular this means that, without loss of generality, $x+y=2z$. Since $(1,0)$, $(0,1)$, and $(-1,-1)$ do not satisfy this relation in any permutation, $x$, $y$, and $z$ cannot all be distinct mod $6\Z^2$.  But if, say, $x$ and $y$ are equal mod $6\Z^2$, there are at least $5$ lattice points on the line between them; only one of these could be $z$ and hence $x$, $y$, and $z$ are not consecutive.
		
		Suppose instead that $x$, $y$, and $z$ are vertices of a triangle.  The area of that triangle is, up to sign, half the determinant of the matrix with columns $y-x$ and $z-x$.  Thus, we need to show that this determinant is never $\pm 1$ or $\pm 2$.  If any two points are equal mod $6\Z^2$ then this determinant is divisible by $6$.  Otherwise, modulo $6$, this determinant (up to sign) is
		\[
		\det \begin{bmatrix} 2 & 1 \\ 1 & 2 \end{bmatrix} = 3 \pmod{6} .
		\]
		Thus the determinant is not $\pm 1$ or $\pm 2$, so $x$, $y$, and $z$ are not contained together in any border triangle.
	\end{proof}

	\section{I-stable Sets}
	\label{sec:IStableSets}
	
	We start our analysis of I-stable sets by giving a construction which
	achieves upper density 1/2.
	
	\begin{const}\label{const:columns}
		For $n\ge 1$, let $I_n \deftobe \Z \times (n\Z)$.  See
		\cref{fig:I2} for $I_2$.
	\end{const}
	
	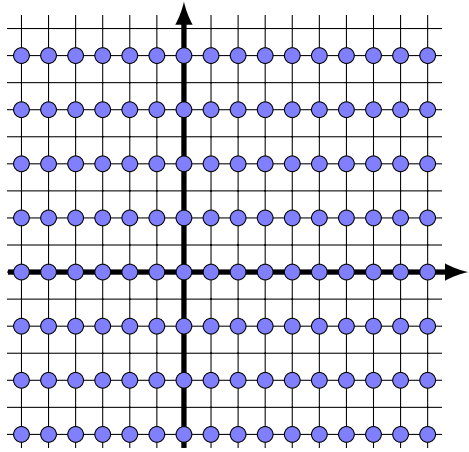
\begin{figure}[h!]
		\centering
		\begin{tikzpicture}[scale=0.36]
			\draw[-latex,line width=2pt] (-0.5,6) -- (16.5,6); 
			\draw[-latex,line width=2pt] (6,-0.5) -- (6,16); 
			%\clip (-0.5,-2.5) rectangle (15.5,15.5);
			%\draw (7,-1) node[below]{$I_2$};
			\clip (-0.5,-0.5) rectangle (15.5,15.5);
			\draw (-1,-1)grid(16,16);
			\begin{scope}[xshift=0cm,yshift=0cm]
				\foreach \x in {0,...,15}{
					\foreach \y in {0,...,8}{
						\draw (\x,2*\y)node[vtx]{};
					}
				}
			\end{scope}
		\end{tikzpicture}
		\caption{%
			The I-stable set $I_{2}$ from \cref{const:columns}. 
		}%
		\label{fig:I2}
	\end{figure}

	\begin{lem}\label{lem:internalcolumns}
		If $n\ne 1,3$, then every internal triangle intersects $I_n$ in at
		most two points.
	\end{lem} 
	\begin{proof}
		From Pick's theorem, it is easy to see that every three points of
		an internal triangle have a convex hull of area $1/2$ or $3/2$.  We
		also recall that the area of the convex hull of $(0,0)$, $(a,b)$
		and $(c,d)$ is
		\begin{equation*}
			\frac{1}{2} 
			\begin{vmatrix} 
				a & b \\
				c & d 
			\end{vmatrix} 
			=
			\frac{ad-bc}{2}.
		\end{equation*}
		Thus the area of the convex hull of any three points of $I_n$ is an
		integral multiple of $\frac{n}{2}$.  Since $n \neq 1,3$, we
		conclude that no internal triangle intersects three points of
		$I_n$.
	\end{proof}
	
	Since every subset of $I_n$ also intersects every internal triangle in
	at most two points, we conclude that every subset of $I_n$ is
	I-stable, which gives the following corollary.
	
	\begin{cor}\label{cor:internal-stable-densities}
		For every $\delta \le 1/2$, there exists an I-stable set $I$ with
		$\del(I)=\delta$.
	\end{cor}
	
	We now work towards proving that every I-stable set has density at
	most $1/2$.  The analogue of \cref{prop:unimodularTriangle} for
	I-stable sets does not hold, and in particular there exist I-stable
	sets that contain a unimodular triangle.  However, these sets turn out
	to be highly structured.  For this we need the following
	constructions.
	
	\begin{const}\label{const:J14_J12}
		Let $\imin \deftobe \setof{(0,0), (1,0), (0,1), (-1,-1)}$.
		Define $\Triangle \deftobe \imin + 4\Z^{2}$ and $\Square
		\deftobe \Triangle\cup (\Triangle+(2,2))$.  See
		\cref{fig:J14_J12}.
	\end{const}
	
	\begin{figure}[h!]
		\centering
		\begin{tikzpicture}[scale=0.36]
			\draw[axisedge] (-1.5,4) -- (15,4); 
			\draw[axisedge] (4,-1.5) -- (4,15); 
			\clip (-1.5,-4) rectangle (14.5,14.5);
			\draw (7,-2) node[below]{$\Triangle$};
			\clip (-1.5,-1.5) rectangle (14.5,14.5);
			\draw (-2,-2)grid(15,15);
			\begin{scope}[xshift=0cm,yshift=0cm]
				\foreach \x in {0,...,3}{
					\foreach \y in {0,...,3}{
						\begin{scope}[xshift=4*\x cm,yshift=4*\y cm]
							\draw (0,0)node[vtx]{};
							\draw (1,0)node[vtx]{};
							\draw (0,1)node[vtx]{};
							\draw (-1,-1)node[vtx]{};
						\end{scope}
				}}
			\end{scope}
		\end{tikzpicture}    
		\hskip 1em  
		\begin{tikzpicture}[scale=0.36]
			\draw[axisedge] (-0.5,4) -- (15,4); 
			\draw[axisedge] (4,-1.5) -- (4,15); 
			\clip (-1.5,-4) rectangle (14.5,14.5);
			\draw (7,-2) node[below]{$\Square$};
			\clip (-0.5,-1.5) rectangle (15.5,15.5);
			\draw (-1,-2)grid(16,16);
			\begin{scope}[xshift=0cm,yshift=0cm]
				\foreach \x in {-1,0,...,3}{
					\foreach \y in {-1,0,...,3}{
						\foreach \s in {0,2}{
							\draw (\s+4*\x,\s+4*\y)node[vtx]{};
							\draw (\s+4*\x+1,\s+4*\y)node[vtx]{};
							\draw (\s+4*\x,\s+4*\y+1)node[vtx]{};
							\draw (\s+4*\x+1,\s+4*\y+1)node[vtx]{};
				}}}
			\end{scope}
		\end{tikzpicture}
		\caption{%
			The sets $\Triangle$ and $\Square$ from \cref{const:J14_J12}.
		}%
		\label{fig:J14_J12}
	\end{figure}
	
	We will see below that $\Triangle$ and $\Square$ are I-stable.
	Indeed, we show that every I-stable set that contains a unimodular
	triangle must contain some unimodular transformation of $\Triangle$;
	containing any other point in addition to $\Triangle$ forces the
	containment of a unimodular transformation of $\Square$; and
	containing any other point in addition to $\Square$ forces the
	I-stable set to be the entire grid $\Z^2$.
	
	\begin{lemma}\label{lem:X1Construction}
		Let $S$ be an I-stable set containing $\setof{(0,0), (0,1),
			(1,0)}$.  Then $\Triangle \subset S$.
	\end{lemma}
	
	\begin{proof}
		Let $S$ be the intersection of all I-stable sets that contain
		$\setof{(0,0), (0,1), (1,0)}$.  Since an intersection of I-stable
		sets is I-stable, $S$ is an I-stable set, and so $S$ is minimal
		under inclusion among the I-stable sets containing $\setof{(0,0),
			(0,1), (1,0)}$.  Note that $(-1,-1) \in S$ since~$S$ contains the
		other three points of the internal triangle $\imin$, so $\imin
		\subset S$.
		
		By the minimality of $S$, if a unimodular map $U$ fixes $\imin$,
		then $U$ also fixes $S$.  Note that the unimodular maps
		\begin{equation*}
			x
			\mapsto
			\begin{bmatrix*}[r] 
				0 & -1 \\ 1 & -1 
			\end{bmatrix*} 
			x 
			\quad \text{ and } \quad 
			x 
			\mapsto
			\begin{bmatrix*}[r] 
				-1 & 1 \\ -1 & 0 
			\end{bmatrix*} 
			x 
		\end{equation*}
		fix $\imin$ and map $\imin + (4, 0)$ to $\imin + (0, 4)$ and $\imin
		+ (-4, -4)$, respectively.  Thus if we show that $S$ contains
		$\imin + (4,0)$, then $S$ must contain $\imin + (0,4)$ and $\imin +
		(-4,-4)$.  By similar reasoning, $S$ must contain $\imin + (0,4) +
		(-4,-4) = \imin + (-4,0)$ and $\imin + (4,0) + (-4,-4) = \imin +
		(0,-4)$, and thus $S$ contains $\Triangle$.
		
		The following inferences demonstrate that $S$ contains $\imin +
		(4,0)$ (see \cref{fig:X1Construction}):
		\begin{align*}
			(0,0),\, (1,0),\, (0,1) \in S & \qquad\limp\qquad (-1, -1) \in S, \\
			(0,0),\, (1,0),\, (0,1) \in S & \qquad\limp\qquad (3, -1),\, (-1, 3) \in S, \\
			(0,0),\, (1,0),\, (-1,-1) \in S & \qquad\limp\qquad (4, 1) \in S, \\
			(0,0),\, (0,1),\, (-1, 3) \in S & \qquad\limp\qquad (1, -4) \in S, \\
			(1,-4),\, (3, -1),\, (4,1) \in S & \qquad\limp\qquad (4, 0) \in S, \\
			(3, -1),\, (4,1),\, (4, 0) \in S & \qquad\limp\qquad (5,0) \in S. \qedhere
		\end{align*}
	\end{proof}
	
	\begin{figure}
		\centering
		\begin{tikzpicture}[scale=0.6]
			\draw[axisedge] (-1.5,0) -- (6.5,0); 
			\draw[axisedge] (0,-4.5) -- (0,4); 
			\clip (-1.5,-4.5) rectangle (5.5,3.5);
			\draw (-2,-5) grid (6,4);
			\begin{scope}[xshift=0cm,yshift=0cm]
				\draw (0,0)node[vtx2]{\phantom{0}};
				\draw (1,0)node[vtx2]{\phantom{0}};
				\draw (0,1)node[vtx2]{\phantom{0}};
				\draw (-1,-1)node[vtx]{1};
				
				\draw (3,-1)node[vtx]{2};
				\draw (-1,3)node[vtx]{2};
				
				\draw (4,1)node[vtx]{3};
				
				\draw (1,-4)node[vtx]{4};
				
				\draw (4,0)node[vtx]{5};
				
				\draw (5,0)node[vtx]{6};
			\end{scope}
		\end{tikzpicture}
		\caption{%
			An I-stable set $S$ containing $\setof{(0,0), (0,1), (1,0)}$ (in
			red) must contain $\imin$ and $\imin + (4,0)$.  The points are
			labeled according to the order in which their membership in $S$
			is demonstrated in the proof of \cref{lem:X1Construction}.
		}%
		\label{fig:X1Construction}
	\end{figure}
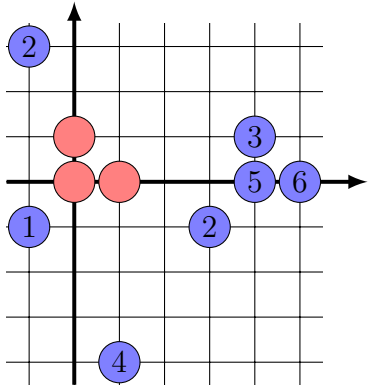
	
	The preceding lemma already shows that any I-stable set containing a
	unimodular triangle contains a unimodular transformation of
	$\Triangle$, since any unimodular triangle is the image of
	$\setof{(0,0), (0,1), (1,0)}$ under some unimodular transformation.
	The following lemma with $\Square$ appears more modest, but see
	\cref{cor:X2IntExtremal} for the analogous conclusion.
	
	\begin{lem}\label{lem:2X2Construction}
		Let $S$ be an I-stable set containing $\setof{(0,0),
			(0,1), (1,0), (1,1)}$.  Then $\Square \subset X$.
	\end{lem}
	\begin{proof}
		We immediately have that
		$\Triangle \subset S$ by \cref{lem:X1Construction}.  In
		addition, $\imin + (2,2) \subset S$ because
		\begin{align*}
			(0,0),\, (1,0),\, (1,1) \in S & \qquad\limp\qquad (2,3) \in S, \\
			(0,0),\, (0,1),\, (1,1) \in S & \qquad\limp\qquad (3,2) \in S, \\
			(1,0),\, (0,1),\, (1,1) \in S & \qquad\limp\qquad (2,2) \in S.
		\end{align*}
		Thus, again by \cref{lem:X1Construction}, we have that
		$\Triangle + (2,2) = \imin + (2,2) + 4\Z^{2} \subset S$.
	\end{proof}
	
	We prove below that $\Triangle$ and $\Square$ are in fact I-stable.
	The case analyses in these proofs will make frequent use of the
	following elementary lemma.  Given vectors $a, b \in \Z^{2}$, we write
	$\det(a, b)$ for the determinant of the $2 \times 2$ matrix with first
	column $a$ and second $b$.
	
	\begin{lem}\label{lem:InternalMinimalCongruences}
		Let $\setof{a, b, c, d} \subset \Z^2$ be an internal triangle, and
		let $a_0 \equiv a$, $b_0 \equiv b$, $c_0 \equiv c$, and $d_0 \equiv
		d$, modulo $4\Z^2$.  Then $a_0 + b_0 + c_0 + d_0 \equiv 0 \mod
		4\Z^2$ and $\det(b_{0}-a_{0}, c_{0}-a_{0}) \cong \pm 1 \mod 4$.
	\end{lem}
	
	Note that the points $a_0, b_0, c_0, d_0$ in
	\cref{lem:InternalMinimalCongruences} are not necessarily
	\textit{uniformly} translated copies of $a,b,c,d$, so they may no longer
	form a minimal triangle.
	
	\begin{proof}[Proof of \cref{lem:InternalMinimalCongruences}]
		Let $T \deftobe \setof{a, b, c, d}$.  It is not difficult to show
		that for any internal triangle, the lattice points that it contains
		are its centroid (\textit{i.e., arithmetic mean}) and its vertices.
		If, without loss of generality, $a$ is the centroid of $T$, then
		$3a = b+c+d$, and so $a_0 + b_0 + c_0 + d_0 \equiv a + b + c + d
		\equiv 0 \mod 4\Z^2$.
		
		To prove the second claim, note that, modulo~$4$, $\det(b_{0} -
		a_{0}, c_{0} - a_{0}) \cong \det(b-a, c-a) \in \setof{\pm 1, \pm
			3}$ since if $a,b,c$ form a unimodular triangle, then the
		determinant is $\pm 1$, and otherwise, $a, b, c$ are the vertices
		of $T$ so the determinant is $\pm 3$.
	\end{proof}
	
	We now prove that $\Triangle$ and $\Square$ are stable.
	
	\begin{prop}\label{prop:X1IntStable}
		The set $\Triangle$ is I-stable with $\del(\Triangle) = 1/4$.
	\end{prop}
	
	\begin{proof}
		Let $T$ be an internal triangle that intersects $\Triangle$ in at
		least three points $a$, $b$, and $c$.  We show that the fourth
		point~$d$ of $T$ is in $\Triangle$.
		
		By construction, the points $a,b,c$ are respectively congruent
		modulo $4\Z^2$ to points $a_0, b_0, c_0$ in~$\imin$.  By
		\cref{lem:InternalMinimalCongruences}, the points $a_0, b_0, c_0$
		are pairwise distinct and $d$ is congruent to a point $d_0$ in
		$[-1,2]^2$ that satisfies the congruence $a_0 + b_0 + c_0 + d_0
		\equiv 0 \mod 4\Z^2$.  By checking each of the four subsets of
		$\imin$ that may equal $\setof{a_0, b_0, c_0}$, one finds in each
		case that this congruence forces $d_0 \in \imin$, and so $d \in
		\Triangle$, as desired.
		
		That $\del(\Triangle) = 1/4$ follows from the fact that $\Triangle$
		contains $4$ of the $16$ points in the window $[-1,2]^{2}$.
	\end{proof}
	
	\begin{prop}\label{prop:X2IntStable}
		The set $\Square$ is I-stable with $\del(\Square)=1/2$.
	\end{prop}
	
	\begin{proof}
		Let $T$ be an internal triangle that intersects $\Square$ in at
		least three points $a$, $b$, and $c$, and let $d$ be the fourth
		point of $T$.  Let $a_{0}, b_{0}, c_{0}, d_{0}$ be the points
		in $[-1,2]^{2}$ that are congruent to
		$a,b,c,d$, respectively, modulo $4\Z^{2}$.
		
		By rotating $T$ by 90 degrees about the rational point
		$(\frac{1}{2}, \frac{1}{2})$ some integer number of times, and then
		adding the vector $(2,2)$ if necessary, we may assume without loss
		of generality that $a_{0} = (0,0)$.  If $b_{0}, c_{0} \in [-1,2]^2
		\cap \Triangle$, then the claim follows from
		\cref{prop:X1IntStable}, so assume without loss of generality that
		$c_{0} \in [-1,2]^2 \cap \Square \setminus \Triangle$; see
		\cref{fig:X2IntStable}.
		
		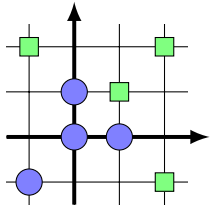
\begin{figure}
			\centering
			\begin{tikzpicture}[scale=0.6]
				\draw[axisedge] (-1.5,0) -- (3,0); 
				\draw[axisedge] (0,-1.5) -- (0,3); 
				\clip (-1.5,-1.5) rectangle (2.5,2.5);
				\draw (-2,-2)grid(3,3);
				\begin{scope}[xshift=0cm,yshift=0cm]
					\draw (0,0)node[vtx4]{};
					\draw (1,0)node[vtx4]{};
					\draw (0,1)node[vtx4]{};
					\draw (-1,-1)node[vtx4]{};
					\draw (2,2)node[gs4vtx]{};
					\draw (-1,2)node[gs4vtx]{};
					\draw (2,-1)node[gs4vtx]{};
					\draw (1,1)node[gs4vtx]{};
				\end{scope}
			\end{tikzpicture}    \\
			\caption{Blue circles indicate the points in $[-1,2]^2 \cap
				\Triangle$.  Green squares indicate the points in $[-1,2]^2 \cap
				\Square \setminus \Triangle$.}
			\label{fig:X2IntStable}
		\end{figure}
		
		If $b_{0} \in [-1,2]^2 \cap \Triangle$, then, after reflecting
		about the line $x = y$ if necessary, we have that $b_{0} \in
		\setof{(1,0), (-1,-1)}$.  If $b_{0} = (1,0)$, then the condition
		that $\det(b_{0}, c_{0}) \cong \pm 1 \mod 4$ forces $c_{0} \in
		\setof{(2, -1), (1,1)}$, while if $b_{0} = (-1, -1)$, then (after
		reflecting about the line $x = y$ if necessary), we likewise have
		that $c_{0} = (2, -1)$.
		
		On the other hand, if $b_{0} \in [-1,2]^2 \cap \Square \setminus
		\Triangle$, then, up to reflection about the line $x=y$, we have
		that $b_{0} \in \setof{(1,1), (2,2), (2,-1)}$.  However, $b_{0} =
		(2,2)$ is impossible because no value of $c_{0} \in [-1,2]^2 \cap
		\Square \setminus \Triangle$ satisfies $\det((2,2), c_{0}) \cong
		\pm 1 \mod 4$, so in fact $b_{0} \in \setof{(1,1), (2,-1)}$ in this
		case.  If $b_{0} = (1,1)$, then, up to reflection about the line
		$x=y$, $c_{0} = (2, -1)$, while if $b_{0} = (2,-1)$, then $c_{0}
		\in \setof{(1,1), (-1,2)}$.
		
		In summary, the following table gives the possible values of
		$b_{0}, c_{0}$ that we must consider, together with the
		corresponding values of $d_{0}$ forced by the condition that $a_{0}
		+ b_{0} + c_{0} + d_{0} \cong 0 \mod 4\Z^{2}$.
		\begin{equation*}
			\begin{array}{lll}
				b_{0}   & c_{0}  & d_{0}   \\
				\hline
				(1,0)   & (2,-1) & (1,1)   \\
				(1,0)   & (1,1)  & (2,-1)  \\
				(-1,-1) & (2,-1) & (-1,2)  \\
				(1,1)   & (2,-1) & (1,0)   \\
				(2,-1)  & (1,1)  & (1,0)   \\
				(2,-1)  & (-1,2) & (-1,-1)
			\end{array}      
		\end{equation*}
		In each case, we find that $d_{0} \in [-1,2]^2 \cap \Square$, so
		that $d \in \Square$, as required.
		
		That $\del(\Square) = 1/2$ follows from the fact that $\Square$
		contains $8$ of the $16$ points in the window $[-1,2]^{2}$.
	\end{proof}
	
	As an aside, this last result implies that there are two very
	different I-stable sets $S$ with $\del(S)=1/2$, namely $\Square$ and
	$I_2$.  This is quite far from the situation for B-stable and
	BI-stable sets, where the only stable sets that we know of which
	achieve the maximum value $\del(S)=1/4$ are subsets of $2\Z^2$ (see
	\cref{quest:borderUnique}).
	
	Combining all of these results gives the following.
	
	\begin{cor}\label{cor:X2IntExtremal}
		Let $S$ be an I-stable set. 
		\begin{enumerate}[label=(\alph*)]
			\item 
			If $S$ contains a unimodular triangle, then $S$ contains a
			unimodular transformation of $\Triangle$.
			
			\item
			If $S$ properly contains $\Triangle$, then $S$ contains a
			unimodular transformation of $\Square$.
			
			\item
			If $S$ properly contains $\Square$, then $S = \mathbb Z^2$.
		\end{enumerate}
		In particular, the only I-stable sets which contain a unimodular
		triangle are unimodular transformations of $\Triangle$, unimodular
		transformations of $\Square$, and $\Z^2$.
	\end{cor}
	
	\begin{proof}
		Observe that $\Triangle$ and $\Square$ are I-stable by
		\cref{prop:X1IntStable,prop:X2IntStable}, respectively.  By
		\cref{lem:X1Construction}, any I-stable set containing a unimodular
		triangle contains a unimodular transformation of $\Triangle$.
		
		We show that if $S$ is an I-stable set containing both $\Triangle$
		and also a point not in $\Triangle$, then $S$ contains a unimodular
		transformation of $\Square$.  Up to translation and reflection, we
		may assume that $S$ contains some $(x,y)$ with $0\leq x\leq y\leq
		3$ not in $\Triangle$, so there are 7 points to consider.  We split
		these into three cases (see \cref{fig:Int_Stable_with_UniTri}).
		
		\begin{figure}[h!]
			\centering
			% INITIAL 1/4 SET
			\begin{tikzpicture}[scale=0.5]
				\draw[axisedge] (-0.5,4) -- (14,4); 
				\draw[axisedge] (4,-0.5) -- (4,14); 
				\clip (-0.5,-0.5) rectangle (13.5,13.5);
				\draw (-2,-2)grid(15,15);
				\begin{scope}[xshift=0cm,yshift=0cm]
					\foreach \x in {0,...,3}{
						\foreach \y in {0,...,3}{
							\begin{scope}[xshift=4*\x cm,yshift=4*\y cm]
								\draw (0,0)node[vtx]{};
								\draw (1,0)node[vtx]{};
								\draw (0,1)node[vtx]{};
								\draw (-1,-1)node[vtx]{};
							\end{scope}
					}}
				\end{scope}
			\end{tikzpicture} \hfill 
			\vspace{1em}
			% ALL 1/2 SET
			\begin{tikzpicture}[scale=0.5]
				\draw[axisedge] (-0.5,4) -- (14,4); 
				\draw[axisedge] (4,-0.5) -- (4,14); 
				\clip (-0.5,-0.5) rectangle (13.5,13.5);
				\draw (-2,-2)grid(15,15);
				\begin{scope}[xshift=0cm,yshift=0cm]
					\foreach \x in {0,...,3}{
						\foreach \y in {0,...,3}{
							\begin{scope}[xshift=4*\x cm,yshift=4*\y cm]
								\draw (0,0)node[vtx]{};
								\draw (1,0)node[vtx]{};
								\draw (0,1)node[vtx]{};
								\draw (-1,-1)node[vtx]{};
								\draw (1,1) node[gsvtx]{};
								\draw (2,2) node[gsvtx]{};
								\draw (3,2) node[gsvtx]{};
								\draw (2,3) node[gsvtx]{};
								\draw (-1,1) node[gs2vtx]{};
								\draw (0,2) node[gs2vtx]{};
								\draw (1,2) node[gs2vtx]{};
								\draw (0,3) node[gs2vtx]{};
								\draw (2,1) node[gs3vtx]{};
								\draw (2,0) node[gs3vtx]{};
								\draw (3,0) node[gs3vtx]{};
								\draw (1,3) node[gs3vtx]{};
							\end{scope}
					}}
				\end{scope}
			\end{tikzpicture}
			% FIRST 1/2 SET
			\begin{tikzpicture}[scale=0.3]
				\draw[axisedge] (-0.5,4) -- (15,4); 
				\draw[axisedge] (4,-0.5) -- (4,15); 
				\clip (-0.5,-0.5) rectangle (13.5,13.5);
				\draw (-2,-2)grid(15,15);
				\begin{scope}[xshift=0cm,yshift=0cm]
					\foreach \x in {0,...,3}{
						\foreach \y in {0,...,3}{
							\begin{scope}[xshift=4*\x cm,yshift=4*\y cm]
								\draw (0,0)node[vtx]{};
								\draw (1,0)node[vtx]{};
								\draw (0,1)node[vtx]{};
								\draw (-1,-1)node[vtx]{};
								\draw (2,1) node[gs3vtx]{};
								\draw (2,0) node[gs3vtx]{};
								\draw (3,0) node[gs3vtx]{};
								\draw (1,3) node[gs3vtx]{};
							\end{scope}
					}}
				\end{scope}
			\end{tikzpicture} \hspace{1cm}
			% SECOND
			\begin{tikzpicture}[scale=0.3]
				\draw[axisedge] (-0.5,4) -- (15,4); 
				\draw[axisedge] (4,-0.5) -- (4,15); 
				\clip (-0.5,-0.5) rectangle (13.5,13.5);
				\draw (-2,-2)grid(15,15);
				\begin{scope}[xshift=0cm,yshift=0cm]
					\foreach \x in {0,...,3}{
						\foreach \y in {0,...,3}{
							\begin{scope}[xshift=4*\x cm,yshift=4*\y cm]
								\draw (0,0)node[vtx]{};
								\draw (1,0)node[vtx]{};
								\draw (0,1)node[vtx]{};
								\draw (-1,-1)node[vtx]{};
								\draw (-1,1) node[gs2vtx]{};
								\draw (0,2) node[gs2vtx]{};
								\draw (1,2) node[gs2vtx]{};
								\draw (0,3) node[gs2vtx]{};
							\end{scope}
					}}
				\end{scope}
			\end{tikzpicture} \hspace{1cm} 
			% THIRD
			\begin{tikzpicture}[scale=0.3]
				\draw[axisedge] (-0.5,4) -- (15,4); 
				\draw[axisedge] (4,-0.5) -- (4,15); 
				\clip (-0.5,-0.5) rectangle (13.5,13.5);
				\draw (-2,-2)grid(15,15);
				\begin{scope}[xshift=0cm,yshift=0cm]
					\foreach \x in {0,...,3}{
						\foreach \y in {0,...,3}{
							\begin{scope}[xshift=4*\x cm,yshift=4*\y cm]
								\draw (0,0)node[vtx]{};
								\draw (1,0)node[vtx]{};
								\draw (0,1)node[vtx]{};
								\draw (-1,-1)node[vtx]{};
								\draw (1,1) node[gsvtx]{};
								\draw (2,2) node[gsvtx]{};
								\draw (3,2) node[gsvtx]{};
								\draw (2,3) node[gsvtx]{};
							\end{scope}
					}}
				\end{scope}
			\end{tikzpicture}
			\caption{%
				Adding a yellow diamond point to the top left picture will
				generate all yellow diamond points, given in the bottom left
				picture, and so on for each color/shape.  Thus, the only
				I-stable proper supersets of $\Triangle$ are the three in the second
				row, which are unimodular transformations of $\Square$.
			}%
			\label{fig:Int_Stable_with_UniTri}
		\end{figure}
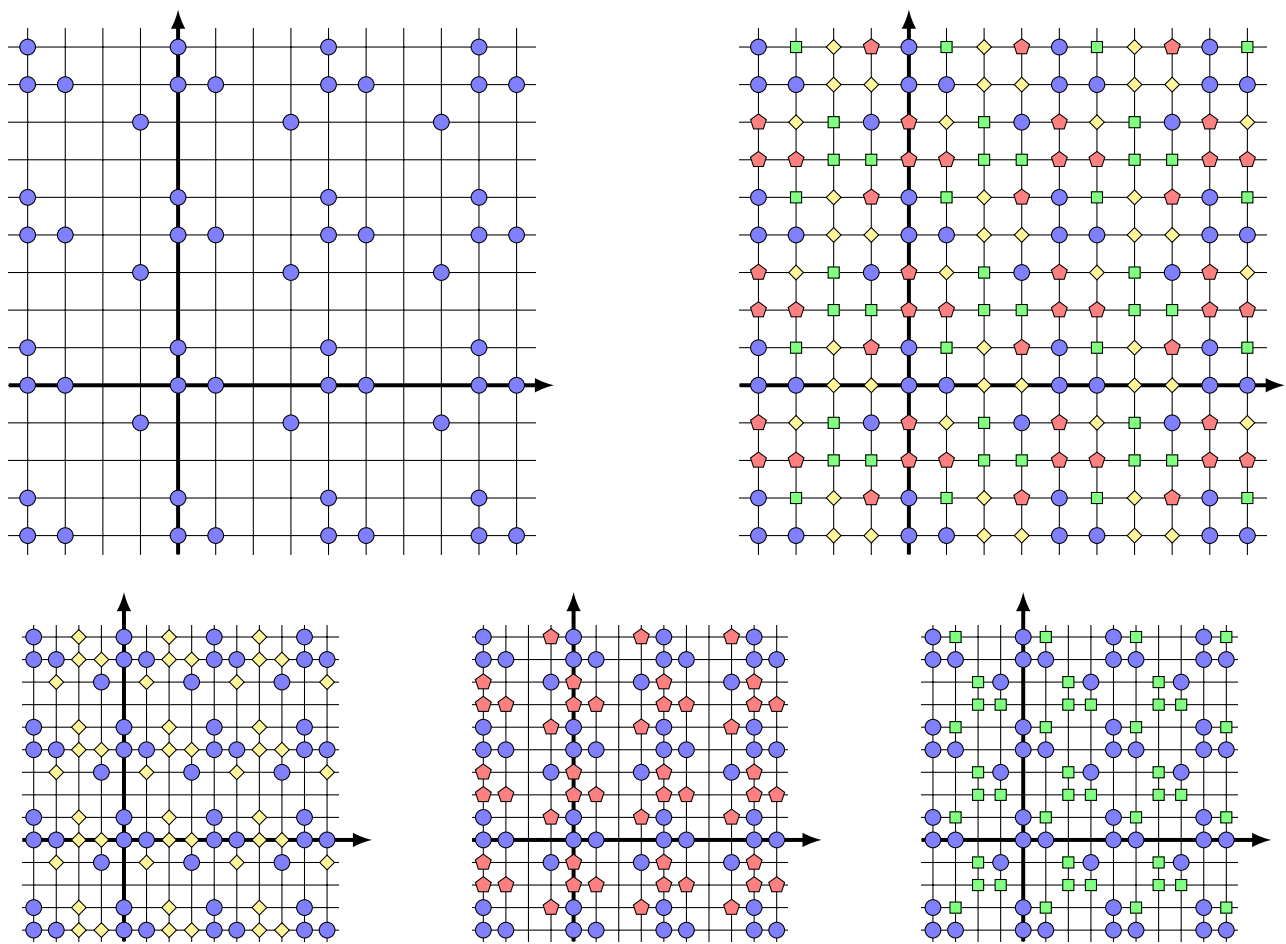
		
		\textbf{Case I:} $S$ contains $(1,3)$.  If $(1,3) \in S$, then $S$
		contains $\{(0,4),(1,3),(0,5),(1,4)\}$, which is the image of
		$\{(0,0),(1,0),(0,1),(1,1)\}$ under a unimodular transformation.
		This is the bottom left picture in
		\cref{fig:Int_Stable_with_UniTri}.
		
		\textbf{Case II:} $S$ contains $(1,2)$, $(0,2)$ or $(0,3)$.  If
		$(1,2) \in S$, then since $(0,1), (3,3) \in S$, we have $(0,2) \in
		S$.  If $(0,2) \in S$, then since $(1,0), (-1,3) \in S$, we have
		$(0,3) \in S$.  If $(0,3) \in S$, then $S$ contains
		$\{(-1,3),(0,3),(0,4),(1,4)\}$, which is the image of
		$\{(0,0),(1,0),(0,1),(1,1)\}$ under a unimodular transformation.
		This is the bottom middle picture in
		\cref{fig:Int_Stable_with_UniTri}.
		
		\textbf{Case III:} $S$ contains $(2,3)$, $(2,2)$, or $(1,1)$.  If
		$(2,3) \in S$, then since $(0,1), (3,3) \in S$, we have $(2,2) \in
		S$.  If $(2,2) \in S$, then since $(0,1), (1,0) \in S$, we have
		$(1,1) \in S$.  If $(1,1) \in S$, then $S$ contains
		$\{(0,0),(1,0),(0,1),(1,1)\}$.  This is the bottom right picture in
		\cref{fig:Int_Stable_with_UniTri}.
		
		In any of these cases, $S$ must contain a unimodular transformation
		of $\Square$ by \cref{lem:2X2Construction}.
		
		Finally, we show that if $S$ is an I-stable set containing
		$\Square$ and also some point not in $\Square$, then $S=\Z^2$.  Up
		to translation and rotation, we may assume that $(2,0) \in S$.
		Then
		\begin{align*}
			(1,0),\, (2,-1),\, (2,0) \in S & \qquad\limp\qquad (3,1) \in S, \\
			(1,0),\, (2,2),\, (3,1)  \in S & \qquad\limp\qquad (2,1) \in S, \\
			(2,-1),\, (4,0),\, (3,1)  \in S & \qquad\limp\qquad (3,0) \in S.
		\end{align*}
		Thus, since $S$ contains $\{(2,0),(3,0),(2,1),(3,1)\}$, we have
		that $S$ contains $\Square + (2,0) = \Z^2 \setminus \Square$.  We
		conclude that $S = \Z^2$ by \cref{lem:2X2Construction}.
	\end{proof}
	
	We summarize these structural observations in the following corollary.
	\begin{cor}\label{cor:structure-internal}
		A set $S \subsetneq \mathbb Z^2$ is I-stable if and only if we have
		one of the following:
		\begin{enumerate}[label=(\alph*)]
			\item
			$S$ is a unimodular transformation of $\Triangle$;
			
			\item
			$S$ is a unimodular transformation of $\Square$;
			
			\item
			$S$ contains at most 2 points from every internal triangle.
		\end{enumerate}
	\end{cor}
	
	\begin{proof}
		By \cref{prop:X1IntStable} and \cref{prop:X2IntStable}, $S$ is
		I-stable in cases (a) and (b); case (c) is immediate.
		
		Conversely, if $S \subsetneq \mathbb Z^2$ is I-stable and contains
		at least 3 points from an internal triangle, then $S$ contains an
		internal triangle, and thereby contains a unimodular triangle.  By
		\cref{cor:X2IntExtremal}, $S$ is a unimodular transformation of
		$\Triangle$ or $\Square$.
	\end{proof}
	
	We now prove our main density result for this section, which together
	with \cref{cor:internal-stable-densities} completes the proof of
	\cref{thm:main}(b).  (We recall that (a) was proven in
	\cref{sec:BStableSets}.)
	
	\begin{prop}\label{prop:main-internal}
		If $S \subsetneq \Z^2$ is I-stable, then $\del(S)\le 1/2$.
	\end{prop}
	
	\begin{proof}
		By \cref{cor:X2IntExtremal}, if $S$ contains a unimodular triangle
		then it is a unimodular transformation of $\Triangle$, $\Square$,
		or $\Z^2$; since $S \neq \Z^2$, we have $\del(S) \leq 1/2$.  If $S$
		does not contain a unimodular triangle, then $S$ intersects every
		$2 \times 2$ square in at most $2$ points.  Tiling $\Z^2$ by $2
		\times 2$ squares, we have $\del(S) \leq 1/2$.
	\end{proof}

	\section{Maximal Stable sets}\label{sec:maximal}
	
	Up to this point we have established a number of structural
	results about stable sets.  \Cref{cor:BorderStableSubsets} states
	that every subset of a B-stable set $S \subsetneq \Z^2$ is
	B-stable and every subset of a BI-stable set $S \subseteq \Z^2$ is
	BI-stable.  Similarly, from \cref{cor:structure-internal}, every
	subset of an I-stable set $S$ is I-stable, unless $S$ is a
	unimodular transformation of $\Triangle$, $\Square$, or $\Z^2$
	(see \cref{fig:Int_Stable_with_UniTri}).  In view of this,
	classifying all stable sets is equivalent to classifying all
	\textit{maximal stable sets}, that is, stable sets $S \subsetneq
	\Z^2$ such that $S$ and $\Z^2$ are the only stable sets containing
	$S$.  In this section we consider some constructions that are
	maximal stable and give some nonconstructive existence results.

	We begin with I-stable sets.  We recall that $I_n \coloneqq \Z \times (n\Z)$ is I-stable for every $n \ne 1,3$ by \cref{lem:internalcolumns}.
	
	\begin{prop}
		Let $I_n \coloneqq \Z \times (n\Z)$. Then $I_n$ is maximal I-stable if and only if either $n=9$ or~$n$~is a prime number other than $3$.
	\end{prop}
	
	\begin{proof}
		If $n\ge 2$ is composite and not $9$, then we can write $n=pq$ with $p\neq 1,3$ and $q>1$. As $I_{pq} \subset I_{p}$, we have that $I_n$ is not a maximal I-stable set in this case. Also, $I_3$ is not I-stable (see \cref{fig:Ik_Int_Stable}). It suffices to prove then that $I_n$ is maximal I-stable for prime $n\neq 3$ and for $n=9$.
		
		Consider the set $X$ consisting of $I_n$ with an extra added point, which we may assume to be of the form $(0,a)$ with $1 \leq a \leq n-1$ without loss of generality. We claim that the only I-stable set containing $X$ is $\Z^2$. Let us first assume $n \neq 3$ is a prime number. Since $\gcd(a,p)=1$, there are integers $t_1, t_2>0$ such that $p t_1 - a t_2 = 1$. Hence, the area of the triangle with vertices $(t_1,0), (t_1-t_2,p), (0,a)$ is 
		\[ \frac{1}{2} \det 
		\begin{bmatrix}
			-t_2 & p \\
			-t_1 & a
		\end{bmatrix}
		= \frac{pt_1 -at_2}{2} = \frac{1}{2},
		\]
		which implies that $\setof{(t_1,0), (t_1-t_2,p), (0,a)}$ is a unimodular triangle with all vertices in $X$ (see \cref{fig:I5_I7}).
		Because $n$ is prime and thus not divisible by $4$, every unimodular transformation of $\Triangle$ and $\Square$ fails to contain $I_n$ (see \cref{fig:Int_Stable_with_UniTri} for the only unimodular transformation of $\Triangle$ and the two unimodular transformations of $\Square$ up to rotation and translation), we conclude by \cref{cor:X2IntExtremal} that the only I-stable set containing X is $\Z^2$. 
		
		Let us now assume $n=9$. The previous argument yields the desired result whenever the additional point $(0,a)$ has $\gcd(a,9)=1$. If $a\in \{3,6\}$, $X$ contains 3 points of an internal triangle, which forces a unimodular triangle (see \cref{fig:Ik_Int_Stable}), and we can conclude this case as before.
	\end{proof}
	
	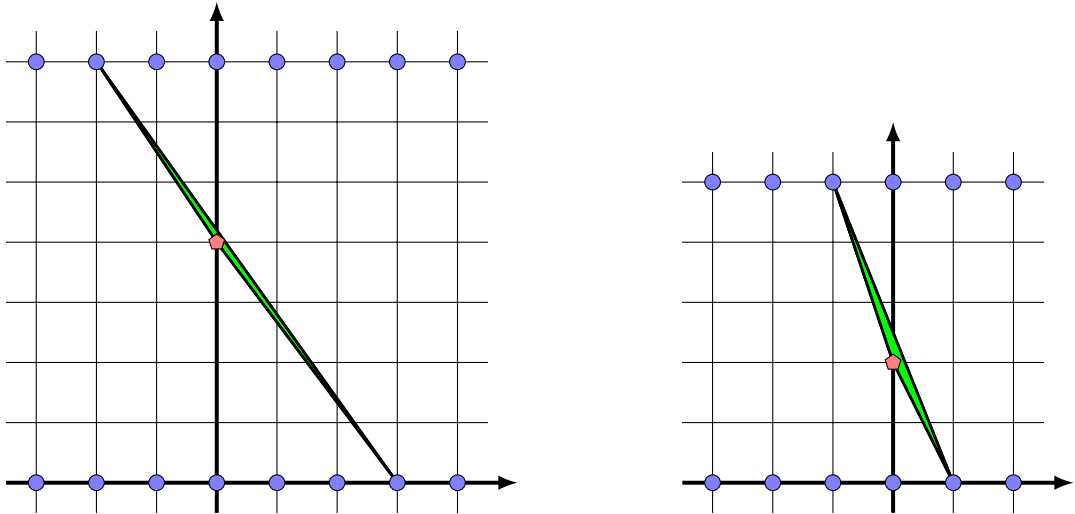
\begin{figure}%[h!]
		\centering
		\begin{tikzpicture}[scale=0.8]
			\draw[axisedge] (-0.5,0) -- (8,0); 
			\draw[axisedge] (3,-0.5) -- (3,8); 
			\clip (-0.5,-0.5) rectangle (7.5,7.5);
			
			\draw (-1,-1)grid(16,16);
			\begin{scope}[xshift=0cm,yshift=0cm]
				\draw 
				(6,0)coordinate (a)
				(1,7)coordinate (b)
				(3,4)coordinate (c)
				;
				\xtriangle{a}{b}{c}
				\draw (3,4) node[gs2vtx]{};
				\foreach \x in {0,...,9}{
					\foreach \y in {0,1}{
						\draw (\x,7*\y)node[vtx]{};
					}
				}
			\end{scope}
		\end{tikzpicture}
		\hspace{2cm}% 
		\begin{tikzpicture}[scale=0.8]
			\draw[axisedge] (-0.5,0) -- (6,0); 
			\draw[axisedge] (3,-0.5) -- (3,6); 
			\clip (-0.5,-0.5) rectangle (5.5,5.5);
			\draw (-1,-1)grid(16,16);
			\begin{scope}[xshift=0cm,yshift=0cm]
				\draw 
				(4,0)coordinate (a)
				(2,5)coordinate (b)
				(3,2)coordinate (c)
				;
				\xtriangle{a}{b}{c}
				\draw (3,2) node[gs2vtx]{};
				\foreach \x in {0,...,9}{
					\foreach \y in {0,...,3}{
						\draw (\x,5*\y)node[vtx]{};
					}
				}
			\end{scope}
		\end{tikzpicture}
		\caption{The sets $I_7 \cup \{(0,4)\}$ and $I_5 \cup \{(0,2)\}$ contain a unimodular triangle.}
		\label{fig:I5_I7}
	\end{figure}
	
	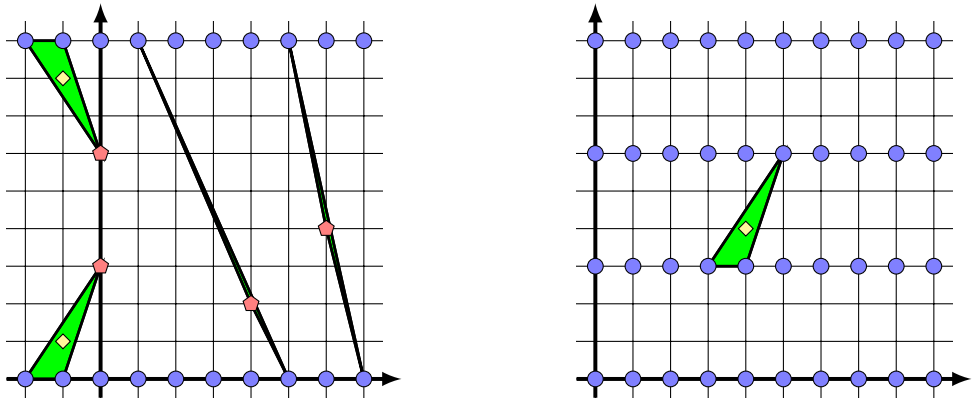
\begin{figure}%[h!]
		\centering
		\begin{tikzpicture}[scale=0.5]
			\draw[axisedge] (-0.5,0) -- (10,0); 
			\draw[axisedge] (2,-0.5) -- (2,10); 
			\clip (-0.5,-0.5) rectangle (9.5,9.5);
			\draw (-1,-1)grid(16,16);
			\begin{scope}[xshift=0cm,yshift=0cm]
				\draw 
				(0,0)coordinate (a)
				(1,0)coordinate (b)
				(2,3)coordinate (c)
				;
				\xtriangle{a}{b}{c}
				\draw (2,3) node[gs2vtx]{};
				\draw (1,1) node[gs3vtx]{};
				
				\draw 
				(0,9)coordinate (a)
				(1,9)coordinate (b)
				(2,6)coordinate (c)
				;
				\xtriangle{a}{b}{c}
				\draw (2,6) node[gs2vtx]{};
				\draw (1,8) node[gs3vtx]{};
				
				\draw 
				(3,9)coordinate (a)
				(7,0)coordinate (b)
				(6,2)coordinate (c)
				;
				\xtriangle{a}{b}{c}
				\draw (6,2) node[gs2vtx]{};
				
				\draw 
				(7,9)coordinate (a)
				(8,4)coordinate (b)
				(9,0)coordinate (c)
				;
				\xtriangle{a}{b}{c}
				\draw (8,4) node[gs2vtx]{};
				
				\foreach \x in {0,...,9}{
					\foreach \y in {0,1}{
						\draw (\x,9*\y)node[vtx]{};
					}
				}
			\end{scope}
		\end{tikzpicture}
		\hspace{2cm} % 
		\begin{tikzpicture}[scale=0.5]
			\draw[axisedge] (-0.5,0) -- (10,0); 
			\draw[axisedge] (0,-0.5) -- (0,10);
			\clip (-0.5,-0.5) rectangle (9.5,9.5);
			\draw (-1,-1)grid(16,16);
			\begin{scope}[xshift=0cm,yshift=3cm]
				\draw 
				(3,0)coordinate (a)
				(4,0)coordinate (b)
				(5,3)coordinate (c)
				;
				\xtriangle{a}{b}{c}
				\draw (4,1) node[gs3vtx]{};
				\foreach \x in {0,...,9}{
					\foreach \y in {-1,0,1,2}{
						\draw (\x,3*\y)node[vtx]{};
					}
				}
			\end{scope}
		\end{tikzpicture}
		\caption{$I_9$ is maximal I-stable, but $I_3$ is not I-stable. The addition of the red pentagonal points to $I_9$ either creates a unimodular triangle or the three vertices of an internal triangle, which forces a unimodular triangle.}
		\label{fig:Ik_Int_Stable}
	\end{figure}
	
	We recall $S_{2/9} \deftobe \setof{(0,0), (1,0)} + 3\Z^2$ from \cref{const:2/9}.
	\begin{prop}
		The set $S_{2/9}$ is maximal B-stable, maximal I-stable, and maximal BI-stable.
	\end{prop}
	
	\begin{proof}
		By \cref{prop:S_2/9}, we know that $S_{2/9}$ is BI-stable.
		Let $X$ be the set obtained after adding one extra point to $S_{2/9}$.
		Then either $X$ contains a unimodular triangle or has 5 points in a row.
		By \cref{prop:unimodularTriangle}, the only B-stable set containing $X$ is $\Z^2$. Therefore, $S_{2/9}$ is maximal B-stable and maximal BI-stable.
		
		Furthermore, every I-stable set containing $X$ also contains a unimodular triangle (see \cref{fig:2/9_Maximal}).
		As every unimodular transformation of $\Triangle$ and $\Square$ fails to contain the set $S_{2/9}$ (see \cref{fig:Int_Stable_with_UniTri} for the only unimodular transformation of $\Triangle$ and the two unimodular transformations of $\Square$ up to rotation and translation), we conclude by \cref{cor:X2IntExtremal} that the only I-stable set containing $X$ is $\Z^2$. 
	\end{proof}
	
	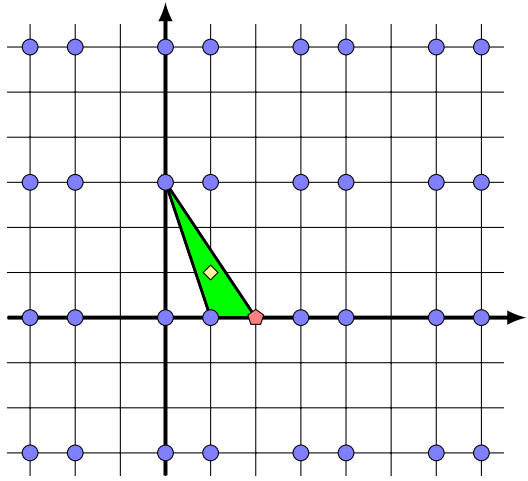
\begin{figure}[h!]
		\centering
		\begin{tikzpicture}[scale=0.6]
			\draw[axisedge] (-0.5,3) -- (11,3); 
			\draw[axisedge] (3,-0.5) -- (3,10); 
			\clip (-0.5,-0.5) rectangle (10.5,9.5);
			\draw (-1,-1)grid(16,16);
			\begin{scope}[xshift=0cm,yshift=0cm]
				\draw 
				(4,3)coordinate (a)
				(5,3)coordinate (b)
				(3,6)coordinate (c)
				;
				\xtriangle{a}{b}{c}
				\draw (5,3) node[gs2vtx]{};
				\draw (4,4) node[gs3vtx]{};
				
				\foreach \x in {0,...,4}{
					\foreach \y in {0,...,4}{
						\draw (3*\x,3*\y)node[vtx]{};
						\draw (3*\x+1,3*\y)node[vtx]{};
					}
				}
			\end{scope}
		\end{tikzpicture}
		\caption{The set $S_{2/9}$ is a maximal I-stable set. The addition of the yellow diamond point creates a unimodular triangle, and the addition of the red pentagonal point forces such a yellow diamond point.}
		\label{fig:2/9_Maximal}
	\end{figure}
	
	All of the constructions we considered in this paper are \textit{periodic}, i.e.\ of the form $X+k\Z^2$ for some finite set $X$ and positive integer $k$. Although we have not explicitly constructed a (maximal) aperiodic stable set, it follows from Zorn's Lemma that such sets must exist as well. For this, we need a simple lemma.
	
	\begin{prop}\label{prop:zorns}
		Let $S \subsetneq \mathbb Z^2$ be B/I/BI-stable. Then there exists a maximal B/I/BI-stable set (of the same stability type as $S$) that contains $S$.
	\end{prop}
	
	\begin{proof}
		We give the proof for I-stable sets. The reasoning is nearly identical for B-stable and BI-stable sets.
		
		Let $\Sigma$ be the collection of all I-stable \textit{proper} subsets of $\mathbb Z^2$ containing $S$. Then $\Sigma$ is nonempty since $S \in \Sigma$. Let $\mathcal{C} \subseteq \Sigma$ be a chain with respect to inclusion, and let $\overline{S} = \bigcup_{T \in \mathcal C} T$.
		
		We claim that $\overline{S}$ is I-stable. Suppose it is not. Then there exist four distinct points $x_1, x_2, x_3, y \in \mathbb Z^2$ forming an internal triangle, with $x_1, x_2, x_3 \in \overline{S}$ and $y \notin \overline{S}$. For each $i = 1,2,3$, let $T_i \in \mathcal C$ contain $x_i$. Since $\mathcal C$ is a chain, we may assume $T_1 \subseteq T_2 \subseteq T_3$ without loss of generality, which means that $x_1, x_2, x_3 \in T_3$. But $y$ cannot be in $T_3$ since $y \notin \overline{S}$, contradicting I-stability of $T_3$.
		
		Next, we show that $\overline{S}$ is a proper subset of $\mathbb 
		Z^2$. If not, then by the same reasoning as in the previous 
		paragraph, some $T \in \mathcal C$ must contain every point in 
		$[4]^2$. Since $T$ is I-stable, we have that $T = \mathbb Z^2$ 
		by \cref{cor:structure-internal} (or \cref{prop:unimodularTriangle} for B-stable and BI-stable sets), contradicting that $T$ is a proper subset of $\mathbb Z^2$.
		
		Thus, $\overline{S} \in \Sigma$ is an upper bound of $\mathcal C$. By Zorn's lemma, $\Sigma$ contains a maximal element.
	\end{proof}
	
	Now we can now give a nonconstructive proof that there exist maximal aperiodic stable sets.
	
	\begin{prop}\label{prop:Iaperiodic}
		There exists a maximal I-stable set $S \sub \Z^2$ that is not periodic. 
	\end{prop}
	\begin{proof}
		Let $S_0 = \{(x,0) \st x \ge 10 \} \cup \{(0,y) \st y \ge 10 \}$. The set $S_0$ is I-stable since any triangle with vertices in $S_0$ has area at least $5$. By \cref{prop:zorns}, $S_0$ is contained in a maximal I-stable set $S$. Suppose that $S$ is periodic. Then the intersection of $S$ with the $x$-axis $\setof{(x, 0) \st x \in \Z}$ is periodic. Since the infinite ray $\setof{(x, 0) \st x \in \Z,\, x \ge 10}$ is in this periodic intersection, all of $\setof{(x, 0) \st x \in \Z}$ is in $S$. Likewise, $S$ contains the entire $y$-axis $\setof{(0, y) \st y \in \Z}$. Thus, $S$ contains the three points $(0,0), (-1,0), (0,-1)$ and hence, by I-stability, $(1,1) \in S$. Now $\{(0,0), (1,0), (0,1), (1,1)\} \subseteq S$, so by \cref{lem:2X2Construction}, $S$ contains $\Square$. But $\Square$ is maximal I-stable by \cref{cor:X2IntExtremal} and does not contain the set $\{(x,0): x \in \Z \} \cup \{(0,y): y \in \Z \} \subseteq S$, which is a contradiction. Thus $S$ is an aperiodic maximal I-stable set.
	\end{proof}
	
	\begin{prop}\label{prop:aperodic}
		There exists a maximal B-stable set $S \sub \Z^2$ that is not periodic. There also exists a maximal BI-stable set $S \sub \Z^2$ that is not periodic.
	\end{prop}
	\begin{proof}
		Let $X \sub \Z$ be any set of integers satisfying the following conditions:
		\begin{enumerate}
			\item[(1)] $X$ contains no three consecutive integers,
			\item[(2)] $X \cup \{x\}$ contains three consecutive integers for all $x \in \Z \setminus X$, and 
			\item[(3)] $X$ is not periodic.
		\end{enumerate}
		Such a set can be constructed as follows.  Let $0.b_1b_2\cdots$ be the binary expansion of any irrational number in $(0,1)$.  Define $x_0 \deftobe 1$ and iteratively $x_{i+1} \deftobe x_i+2+b_i$.  Take $X' \deftobe \{x_i\}_{i \ge 0}\cup \{x_i+1\}_{i \st b_i=1}$ and set $X \deftobe X' \cup (-X')$.  It is straightforward to check that this set has the desired properties.
		
		Let $S_0 \deftobe X \times \{0\} \sub \Z^2$.  By Condition (1), 
		$S_0$ intersects each border triangle in at most two points, so 
		it is a B-stable proper subset of $\Z^2$.  By \cref{prop:zorns}, 
		there exists some $S \supseteq S_0$ that is maximal B-stable.  
		Observe that $S \cap (\Z\times \{0\}) = X$, as otherwise $S$ would contain three consecutive integers by Condition (2), which by \cref{prop:unimodularTriangle} would imply $S = \Z^2$.  Finally, $S$ is not periodic because $X$ is not periodic by Condition (3).
		
		For a maximal BI-stable set, we apply \cref{prop:zorns} to obtain $S \supseteq S_0$ that is maximal BI-stable, and the rest of the argument is unchanged.
	\end{proof}

	\section{Concluding Remarks and Maximal Stable Sets}\label{sec:ConcludingRemarks}
	In this paper we defined $S \sub \Z^2$ to be B/I/BI-stable if there exists no border/internal/minimal triangle (respectively) on four points that intersects $S$ in exactly three points.  There are many ways one could generalize these notions.  When considering the number-theoretic motivation, work in $\Z^2$ is related to degree-$3$ number fields.  To handle higher-degree fields, it is natural to consider a notion of stability with respect to minimal simplices in~$\Z^d$.  However, the classification of simplices in $\Z^d$ containing $d+2$ points is considerably more complicated when $d \ge 3$ than when $d=2$. For example, there exist infinitely many tetrahedra in $\Z^{3}$, not equivalent up to unimodular transformation, containing exactly five points, including eight tetrahedra containing one interior point \cite[Table~1]{BlaSan2016}.  
	
	As discussed in \cref{sec:maximal}, understanding stable sets is equivalent to understanding maximal stable sets.  As such, many of our open problems center around such sets.  For example, we showed that any B-stable set $S \subsetneq \Z^2$ has $\del(S) \le 1/4$, and that this is tight by considering $S = 2\Z^2$.  The following asks if this is essentially the only such set with this property.
	\begin{quest}\label{quest:borderUnique}
		Does there exist a set $S \subsetneq \Z^2$ that is maximal B-stable with $\del(S)=1/4$ and that is not equal to a unimodular transformation of $2\Z^2$?
	\end{quest}
	A negative answer to this question would also imply a negative answer to the analogous question for BI-stable sets.  Here the condition that $S$ be maximal is essential, as otherwise one can just consider $2\Z^2$ after deleting a point.  A related question is the following.
	\begin{quest}\label{quest:conscutive}
		What is the largest value of $\del(S)$  for B-stable $S\subsetneq \Z^2$ if $S$ contains $(0,0)$ and $(1,0)$?
	\end{quest}
	The set $S_{2/9}$ shows that $\del(S)=2/9$ is achievable, and \cref{prop:main-border} implies $\del(S)\le 1/4$.  If this latter bound were tight, then this would imply a positive answer to \cref{quest:borderUnique}, and any proof showing $\del(S)=1/4$ is not possible might give insight into proving that \cref{quest:borderUnique} has a negative answer.
	
	It would be interesting if one could construct a stable set $S$ with positive density that is not periodic. 
	\begin{quest}\label{quest:aperiodic}
		Does there exist an aperiodic set $S \sub \Z^2$ with $\del(S)>0$ that is maximal B/I/BI-stable?
	\end{quest}
	We note that if either $\del(S)=0$ or $S$ is periodic, then $\del(\phi(S))=\del(S)$ for any unimodular transformation $\phi$.  As such, a negative answer to this question would imply that $\del(S)$ is an ``intrinsically geometric'' property of maximal stable sets.  The maximality condition in \cref{quest:aperiodic} is critical, as $2\Z^2$ minus a point trivially satisfies all of the other conditions.  Similarly, \cref{prop:Iaperiodic,prop:aperodic} show that we can find such sets if one drops the requirement $\del(S)>0$. 
	
	\section{Acknowledgements}
	This work was started at the 2022 Graduate Research Workshop in Combinatorics, which was supported in part by NSF grant 1953985 and a generous award from the Combinatorics Foundation. We thank the referees for their careful reading and useful comments that improved the presentation of this manuscript.

	\bibliographystyle{abbrvurl}
	\bibliography{refs.bib}

	\newpage 
	\appendix
	\section{Computer-assisted proofs}
	\label{app:code}
	
	In this appendix, we discuss the Python code for a program that proves \cref{lem:6x6} and \cref{lem:consec16}, and classifies the sets considered in \cref{rem:6x12}. This code can be found\locationofprograms 
	
	The code enumerates all B-stable subsets of a given set of points using a depth-first search. First we set up some helpful functions, including the main recursive step. Then we give the three applications. 
	\cref{lem:6x6} states every B-stable set $S \subsetneq \Z^2$ has $|S \cap ([6]\times [6])|\le 9$ and \cref{rem:6x12} describes all B-stable sets $S \subseteq [6]^2$ with $\Gamma(S) \neq \emptyset$ and $|S|=9$ up to rotation and reflection; the code completes this calculation in only a few seconds on a personal computer. \cref{lem:consec16} states that every B-stable set $S \subsetneq \Z^2$ with $\Gamma(S) \cap ([6]\times [12]) \neq \emptyset$ has $|S \cap ([6]\times [12])|\le 16$; this takes longer to calculate, about 6 minutes on the same hardware.
	
	\begin{figure}[h!]
		\centering
		\begin{tikzpicture}[scale=0.8]
			\begin{scope}[xshift=0cm,yshift=0 cm]
				\clip (-0.5,-0.5) rectangle (5.5,5.5);
				\draw (-1,-1)grid(6,6);
				
				\foreach \x in {(0,5), (1,5), (3,5), (5,5), (0,2), (5,2), (1, 0), (3,0), (5,0)} {    
					\draw \x node[vtx]{};
				}
			\end{scope}
			\begin{scope}[xshift=7cm,yshift=0 cm]
				\clip (-0.5,-0.5) rectangle (5.5,5.5);
				\draw (-1,-1)grid(6,6);
				
				\foreach \x in {(0,5), (2,5), (3,5), (5,5), (0,2), (2,2), (5, 2), (0,0), (5,0)} {    
					\draw \x node[vtx]{};
				}
			\end{scope}
			\begin{scope}[xshift=14cm,yshift=0 cm]
				\clip (-0.5,-0.5) rectangle (5.5,5.5);
				\draw (-1,-1)grid(6,6);
				
				\foreach \x in {(3,0), (2,5), (3,5), (5,5), (0,2), (2,2), (5, 2), (0,0), (5,0)} {    
					\draw \x node[vtx]{};
				}
			\end{scope}
		\end{tikzpicture}
		\vspace{0.5cm}
		
		\begin{tikzpicture}[scale=0.8]
			\begin{scope}[xshift=0cm,yshift=0 cm]
				\clip (-0.5,-0.5) rectangle (5.5,5.5);
				\draw (-1,-1)grid(6,6);
				
				\foreach \x in {(0,5), (2,5), (3,5), (5,5), (0,2), (2,0), (5, 2), (0,0), (4,0)} {    
					\draw \x node[vtx]{};
				}
			\end{scope}
			
			\begin{scope}[xshift=8cm,yshift=0 cm]
				\clip (-0.5,-0.5) rectangle (5.5,5.5);
				\draw (-1,-1)grid(6,6);
				
				\foreach \x in {(0,5), (2,5), (3,5), (5,5), (0,2), (2,0), (5, 2), (0,0), (5,0)} {    
					\draw \x node[vtx]{};
				}
			\end{scope}
		\end{tikzpicture}
		\caption{All B-stable sets $S \subset [6]^2$ with $\Gamma(S) \neq \emptyset$ and $|S|=9$, up to reflection and rotation.}
		\label{fig:9pts}
	\end{figure}
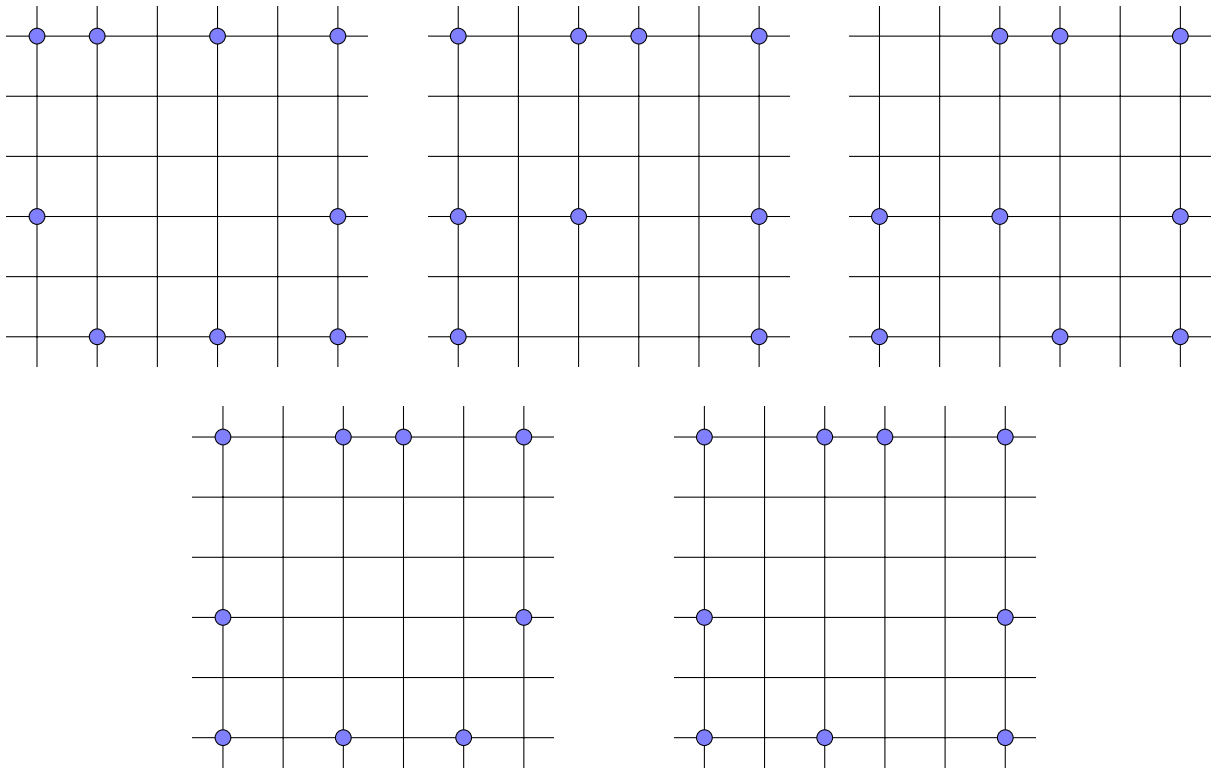

\newpage
\section{Computer-free proof of \cref{lem:6x6}}
\label{app:6x6bound}

In this Appendix, we prove \cref{lem:6x6}, namely that every B-stable set $S \subsetneq \Z^2$ intersects $[6]^2$ in at most 9 points, without computer aid. A proof along similar lines should be possible for \cref{lem:consec16}, but we do not provide one here.

First, we show that it is enough to prove the following.

\begin{lem} \label{lem:6x6app1}
	Let $S \subsetneq \Z^2$ be a B-stable set. If $S \cap [6]^2$ contains a pair of points $(x,y)$ and $(x+1,y)$, then $|S \cap [6]^2| \le 9$.
\end{lem}

\begin{lem} \label{lem:6x6app2}
	If $S\subsetneq \Z^2$ is B-stable, $S \cap [6]^2$ has no pair of points at distance $1$ and $S \cap [6]^2$ contains two points at distance $\sqrt{2}$, then $|S \cap [6]^2|\le 9$.
\end{lem}

\begin{proof}[Proof of \cref{lem:6x6}]
	Assume that $S \subsetneq \Z^2$ is a B-stable set and $|S \cap [6]^2| > 9$. Then, by \cref{lem:6x6app1} applied to $S$ and a 90-degree rotation of $S$, we have that $S \cap [6]^2$ contains no pair of points at distance 1. Thus, by \cref{lem:6x6app2}, $S \cap [6]^2$ contains no pair of points at distance $\sqrt{2}$.
	This implies that every translate of $[2]\times [2]$ intersects $S$ in at most one point, contradicting the fact that $|S \cap [6]^2| > 9$.
\end{proof}

\subsection{Proof of \cref{lem:6x6app1}}
Let $T = S \cap [6]^2$.
If such $T$ contains a pair of horizontally consecutive points in the 2nd, 3rd, 4th, or 5th rows of $[6]^2$, then $T$ contains at most 8 points (see \cref{fig:6x6_middle_rows}). We may assume this is not the case, and by symmetry that $T$ also contains no pair of vertically consecutive points in the middle four columns.

\begin{figure}[h!]
	\centering
	\begin{tikzpicture}[scale=0.8]
		\draw (0,0)grid(5,5);
		
		\begin{scope}[xshift=0cm,yshift=4cm]
			\draw
			(0,0)coordinate (a1)
			(1,0)coordinate (a2)
			(2,0)coordinate (a3)
			;
			\foreach \x in {1,2}{
				\draw (a\x)node[vtx]{};
			}
		\end{scope}
		
		\foreach \y in {2,3,5}{
			\begin{scope}[xshift=0cm,yshift=\y cm]
				\foreach \x in {0,...,5}{
					\draw (\x,0)node[novtx]{};
				}
			\end{scope}
		}
		
		\foreach \x in {0,3}{
			\begin{scope}[xshift=\x cm,yshift=0cm]
				\draw 
				(0,1)coordinate (a)
				(2,0)coordinate (b)
				;
				\xwindow{a}{b}{2}
			\end{scope}
		}
		
		\begin{scope}[xshift=3 cm,yshift=4cm]
			\draw 
			(-1,0)coordinate (a)
			(2,0)coordinate (b)
			;
			\xwindow{a}{b}{2}
		\end{scope}
		
	\end{tikzpicture}
	\caption{Possible sets in which $S \cap [6]^2$ contains a pair of points at distance $1$ in the 2nd, 3rd, 4th, or 5th row. Points marked with a red X cannot be included in $S$, or else $S=\mathbb{Z}$.}
	\label{fig:6x6_middle_rows}
\end{figure}
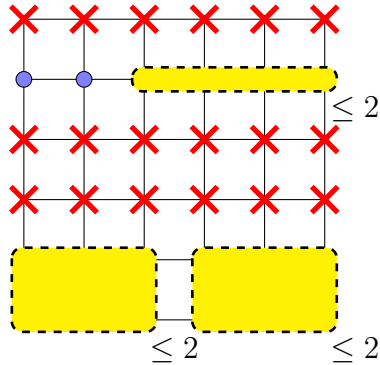

Now without loss of generality, the top row of $[6]^2$ contains a consecutive pair from $T$. Then as we can see from \cref{fig:6x6_1st_row}, $T$ contains at most 4 points from row 6, no points from rows 5 and 4, and at most 6 points total from rows 1, 2, and 3, and hence $|T| \le 10$. Thus either $|T|\leq 9$ already, or the top row must contain exactly 4 points of $T$, and in particular, we can assume without loss of generality that either $(1,6), (2,6) \in T$, $(5, 6), (6,6) \in T$, or $(1,6), (3,6), (4,6), (6,6) \in T$.

\begin{figure}[h!] 
	\centering
	\begin{tikzpicture}[scale=0.8]
		\draw (0,0)grid(5,5);
		
		\begin{scope}[xshift=0cm,yshift=5cm]
			\draw
			(0,0)coordinate (a1)
			(1,0)coordinate (a2)
			(2,0)coordinate (a3)
			;
			\foreach \x in {1,2}{
				\draw (a\x)node[vtx]{};
			}
			\foreach \x in {3}{
				\draw (a\x)node[novtx]{};
			}
		\end{scope}
		
		\begin{scope}[xshift=0cm,yshift=3cm]
			\foreach \x in {0,...,5}{
				\foreach \y in {0,1}
				\draw (\x,\y)node[novtx]{};
			}
		\end{scope}
		
		\foreach \x in {0,2,4}{
			\begin{scope}[xshift=\x cm,yshift=0cm]
				\draw 
				(0,2)coordinate (a)
				(1,0)coordinate (b)
				;
				\xwindow{a}{b}{2}
			\end{scope}
		}
		
		\begin{scope}[xshift=3cm,yshift=5cm]
			\draw 
			(0,0)coordinate (a)
			(2,0)coordinate (b)
			;
			\xwindow{a}{b}{2}
		\end{scope}
		
	\end{tikzpicture}
	\hskip 4em
	\begin{tikzpicture}[scale=0.8]
		\draw (0,0)grid(5,5);
		
		\begin{scope}[xshift=2cm,yshift=5cm]
			\draw
			(0,0)coordinate (a1)
			(1,0)coordinate (a2)
			(3,0)coordinate (a3)
			(-2,0)coordinate (a4)
			;
			\draw (-2,0)node[vtx]{};
			\draw (0,0)node[vtx]{};
			\draw (1,0)node[vtx]{};
			\draw (3,0)node[vtx]{};
			\draw (-1,0)node[novtx]{};
			\draw (2,0)node[novtx]{};
		\end{scope}
		
		\foreach \y in {3,4}{
			\begin{scope}[xshift=0cm,yshift=\y cm]
				\foreach \x in {0,...,5}{
					\draw (\x,0)node[novtx]{};
				}
			\end{scope}
		}
		
		\foreach \x in {0,2,4}{
			\begin{scope}[xshift=\x cm,yshift=0cm]
				\draw 
				(0,2)coordinate (a)
				(1,0)coordinate (b)
				;
				\xwindow{a}{b}{2}
			\end{scope}
		}
		
	\end{tikzpicture}
	\caption{Possible sets with two points at distance 1 on the 1st row of a $6\times 6$ window.}
	\label{fig:6x6_1st_row}
\end{figure}
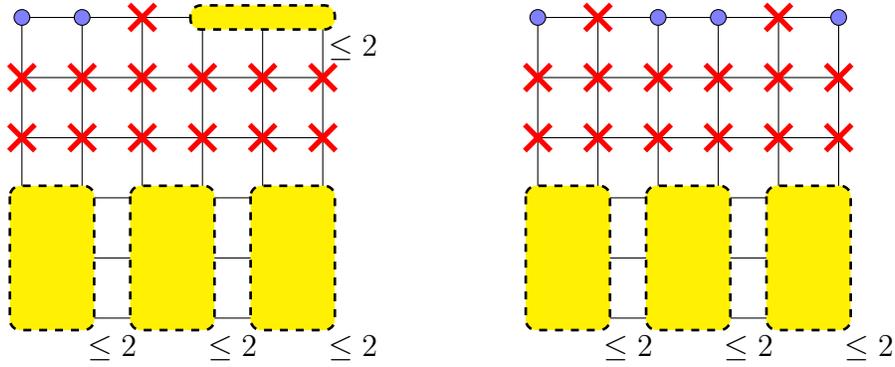

We will break these into further cases depending on which 2 points appear in the `middle $2\times 3$ window' with coordinates $\{3,4\} \times \{1,2,3\}$, see \cref{fig:15cases} for all 15 possibilities. 

\begin{figure}[h!] 
	\centering
	\begin{tikzpicture}[scale=0.8]
		\foreach \i in {0,3} {
			\foreach \j[evaluate={\k=int(\i+\j);}] in {0,1,2} {
				\coordinate (A\k) at (\i/3, \j);
			}
		}
		
		\edef\shiftera{0}
		\edef\shifterb{0}
		
		\newcounter{step}
		\setcounter{step}{1}
		
		\foreach \i[evaluate={\jstart=int(\i+1);}] in {0,...,4} {
			\foreach \j in {\jstart,...,5} {
				\foreach \k in {0,...,5} {
					\coordinate (B\k) at ($(A\k)+(\shiftera,-\shifterb)$);
				}
				
				\draw ($(B0)-(0.01,0)$)grid(B5);
				
				\foreach \k in {0,...,5}{
					\ifthenelse{\NOT \equal{\k}{\i} \AND \NOT \equal{\k}{\j}}{\draw (B\k) node[novtx]{};}{}
				}
				
				\draw (B\i)node[vtx]{};
				\draw (B\j)node[vtx]{};
				
				\draw ($(B0)+(0.5,-0.3)$) node[below]{Case \thestep};
				\addtocounter{step}{1}
				
				\pgfmathparse{\shiftera+3}
				\xdef\shiftera{\pgfmathresult}
				\ifthenelse{\i=0 \AND \j=5} {
					\pgfmathparse{\shifterb+4}
					\xdef\shifterb{\pgfmathresult}
					\xdef\shiftera{0}
				}{}
				\ifthenelse{\i=2 \AND \j=3} {
					\pgfmathparse{\shifterb+4}
					\xdef\shifterb{\pgfmathresult}
					\xdef\shiftera{0}
				}{}
			}
		}
	\end{tikzpicture}
	\caption{All possible configuration of 2 points in a $2\times 3$ window.}
	\label{fig:15cases}
\end{figure}
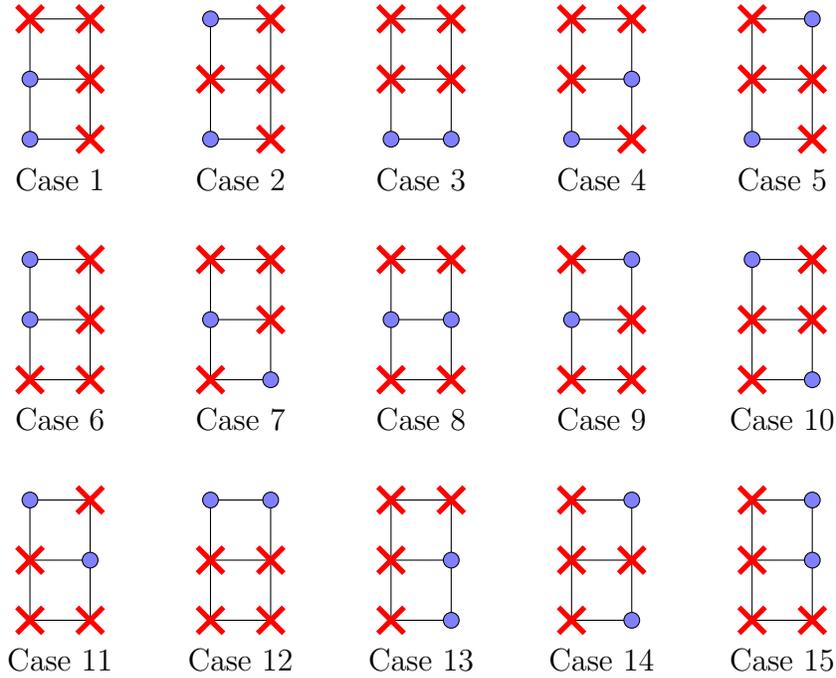

From now on, we will refer to each possibility as the label included in \cref{fig:15cases}. By our assumption on consecutive points in middle rows and columns, we can immediately rule out cases 1, 6, 8, 12, 13, and 15. Moreover in case 3, regardless of the configuration in the top row, we can have at most 4 points in the last three rows for a total of at most $8$.

The configuration $(1,6), (3,6), (4,6), (6,6) \in T$ is then relatively straightforward.  In cases 5, 9, 10, and 11, $T$ has three points of a unimodular triangle. In cases 2, 4, 7, and 14, $T$ has three points of a border triangle. 

Finally, we handle the configuration $(1,6), (2,6) \in T$, which by symmetry, covers the configuration $(5, 6), (6,6) \in T$ as well.
In cases 10 and 11, there are three points of a unimodular triangle. In cases 2 and 7, there are three points of a border triangle. We deal with the remaining cases of 4, 5, 9, and 14 in the following. 

In cases 4, 5, and 9, we rule out new points that will create either unimodular or border triangle with the points in $T$. A quick check shows that $T$ has at most 9 points in these cases, as seen in the following. 

\begin{center}
	\begin{tikzpicture}[scale=0.8]
		\draw (0,0)grid(5,5);
		
		\begin{scope}[xshift=0cm,yshift=5cm]
			\draw
			(0,0)coordinate (a1)
			(1,0)coordinate (a2)
			(2,0)coordinate (a3)
			;
			\foreach \x in {1,2}{
				\draw (a\x)node[vtx]{};
			}
			\foreach \x in {3}{
				\draw (a\x)node[novtx]{};
			}
		\end{scope}
		
		\begin{scope}[xshift=0cm,yshift=0cm]
			\draw
			(0,2)coordinate (a1)
			(1,0)coordinate (a2)
			;
			\xwindow{a1}{a2}{2}
		\end{scope}
		
		\begin{scope}[xshift=0cm,yshift=3cm]
			\foreach \x in {0,...,5}{
				\foreach \y in {0,1}
				\draw (\x,\y)node[novtx]{};
			}
		\end{scope}
		
		\begin{scope}[xshift=1cm,yshift=0cm]
			\foreach \x in {0,...,4}{
				\foreach \y in {0,1}
				\draw (\x,\x+\y)node[novtx]{};
			}
			\draw (-1,0)node[novtx]{};
		\end{scope}
		
		\begin{scope}[xshift=2cm,yshift=0cm]
			\foreach \x in {0,1}{
				\draw (\x,\x)node[vtx]{};
			}
			\draw (2,2)node[novtx]{};
		\end{scope}
		
		\begin{scope}[xshift=3cm,yshift=0cm]
			\foreach \x in {0,1}{
				\foreach \y in {0,1}
				\draw (\x+\y,\x)node[novtx]{};
			}
			\draw (2,2)node[novtx]{};
		\end{scope}
		
		\begin{scope}[xshift=7 cm,yshift=0cm]
			\draw (0,0)grid(5,5);
			
			\begin{scope}[xshift=0cm,yshift=5cm]
				\draw
				(0,0)coordinate (a1)
				(1,0)coordinate (a2)
				(2,0)coordinate (a3)
				;
				\foreach \x in {1,2}{
					\draw (a\x)node[vtx]{};
				}
				\foreach \x in {3}{
					\draw (a\x)node[novtx]{};
				}
			\end{scope}
			
			\draw
			(0,2)coordinate (x1)
			(1,0)coordinate (x2)
			(4,2)coordinate (x3)
			(5,0)coordinate (x4)
			;
			\xwindow{x1}{x2}{2}
			\xwindow{x3}{x4}{2}
			
			\begin{scope}[xshift=0cm,yshift=3cm]
				\foreach \x in {0,...,5}{
					\foreach \y in {0,1}
					\draw (\x,\y)node[novtx]{};
				}
			\end{scope}
			
			\begin{scope}[xshift=2cm,yshift=0cm]
				\foreach \x in {0,1}{
					\draw (\x,2*\x)node[vtx]{};
				}
			\end{scope}
			
			\foreach \y in {1,3}{
				\begin{scope}[xshift=\y cm,yshift=0cm]
					\foreach \x in {0,1,2}{
						\draw (\x,2*\x)node[novtx]{};
					}
				\end{scope}
			}
			
			\foreach \y in {2,3}{
				\begin{scope}[xshift=\y cm,yshift=1cm]
					\foreach \x in {0,1,2}{
						\draw (\x,2*\x)node[novtx]{};
					}
				\end{scope}
			}
		\end{scope}
		
		\begin{scope}[xshift=14 cm,yshift=0cm]
			\draw (0,0)grid(5,5);
			
			\begin{scope}[xshift=0cm,yshift=5cm]
				\draw
				(0,0)coordinate (a1)
				(1,0)coordinate (a2)
				(2,0)coordinate (a3)
				;
				\foreach \x in {1,2}{
					\draw (a\x)node[vtx]{};
				}
				\foreach \x in {3}{
					\draw (a\x)node[novtx]{};
				}
			\end{scope}
			
			\begin{scope}[xshift=0cm,yshift=3cm]
				\foreach \x in {0,...,5}{
					\foreach \y in {0,1}
					\draw (\x,\y)node[novtx]{};
				}
			\end{scope}
			
			\begin{scope}[xshift=0cm,yshift=0cm]
				\foreach \x in {0,...,4}{
					\foreach \y in {0,1}
					\draw (\x,\x+\y)node[novtx]{};
				}
				\draw (5,5)node[novtx]{};
			\end{scope}
			
			\begin{scope}[xshift=2cm,yshift=1cm]
				\foreach \x in {0,1}{
					\draw (\x,\x)node[vtx]{};
				}
				\draw (-1,-1)node[novtx]{};
			\end{scope}
			
			\begin{scope}[xshift=2cm,yshift=0cm]
				\foreach \x in {0,1}{
					\foreach \y in {0,1,2}{
						\draw (\x+\y,\y)node[novtx]{};
				}}
			\end{scope}
			
			\begin{scope}[xshift=4cm,yshift=0cm]
				\draw
				(0,0)coordinate (a)
				(1,0)coordinate (b)
				(1,1)coordinate (c)
				;
				\xtriwindow{a}{b}{c}{2}
			\end{scope}
		\end{scope}
		
	\end{tikzpicture}
\end{center}

For case 14, we will need a slightly more detailed analysis. For this, we introduce a new labeling system in pictures that will help the case checking later on.

The points $(2,6)$ and $(4,3)$ are labeled 1 and 2 in the following figure. They form a border triangle with the bottom-right point $(6,1)$, which forbids $(6,1)$ from being in $T$. We notate this in the picture with a red $1,2$ label over that point. Further, the points $(4,1)$ and $(4,3)$ make it impossible to have both points $(6,2)$ and $(6,3)$ from the `right $2\times 3$ window' $\{5,6\} \times [3]$. We conclude that there is at most one point from $T$ in that $2\times 3$ window. Thus we conclude that in this last case, as in all others, $T$ has at most 9 points. 

\begin{center}
	\begin{tikzpicture}[scale=0.8]
		\draw (0,0)grid(5,5);
		
		\begin{scope}[xshift=0cm,yshift=5cm]
			\draw
			(0,0)coordinate (a1)
			(1,0)coordinate (a2)
			(2,0)coordinate (a3)
			;
			\foreach \x in {1,2}{
				\draw (a\x)node[vtx]{};
			}
			\foreach \x in {3}{
				\draw (a\x)node[novtx]{};
			}
		\end{scope}
		
		\begin{scope}[xshift=0cm,yshift=3cm]
			\foreach \x in {0,...,5}{
				\foreach \y in {0,1}
				\draw (\x,\y)node[novtx]{};
			}
		\end{scope}
		
		\begin{scope}[xshift=3cm,yshift=0cm]
			\foreach \x in {0,2}{
				\draw (0,\x)node[vtx]{};
			}
			\draw (0,1)node[novtx]{};
		\end{scope}
		
		\foreach \y in {2,4}{
			\begin{scope}[xshift=\y cm,yshift=0cm]
				\foreach \x in {0,...,5}{
					\draw (0,\x)node[novtx]{};
				}
			\end{scope}
		}
		
		\draw (0,5)node[right, blue]{};
		\draw (1,5)node[right, blue]{$1$};
		\draw (3,2)node[right, blue]{$2$};
		\draw (3,0)node[right, blue]{};
		
		\draw (5,0)node[red,yshift=-0.1cm]{$1,2$};
		
		\begin{scope}[xshift=5 cm,yshift=1cm]
			\draw 
			(0,1)coordinate (a)
			(0,0)coordinate (b)
			;
			\xwindow{a}{b}{1}
		\end{scope}
		
	\end{tikzpicture}
\end{center}

\subsection{Proof of \cref{lem:6x6app2}}

We now assume that $S$ has no points at distance $1$ but has points at distance $\sqrt{2}$. 

Throughout this argument, in the grid $[k]^2$, the \emph{main diagonal} is  the set of $k$ lattice points $(x,y)$ for which $x-y=0$; the \emph{main antidiagonal} is the set of points for which $x+y=k$. Similarly, the \emph{$d^\text{th}$ diagonals} are the sets of $2(k-d)$ lattice points in $[k]$ with $x-y=\pm d$; and the \emph{$d^\text{th}$ antidiagonals} are the ones with $x+y=k\pm d$.

We will make use of the following Claim (see \cref{fig:claim}).

\begin{claim} \label{claim:appendix}
	If $S$ has no points at distance 1, then an upper or lower triangular part of the grid $[k]^2$ can contain at most 3 points when $k=3$, and at most 6 points when $k=5$. The only configurations achieving equality are shown in \cref{fig:equality_claim}. Moreover, when $k=4$, if we further assume there are no points at distance $\sqrt{2}$ on the main diagonal or antidiagonal, there are at most 3 points on such a triangular grid. 
\end{claim}

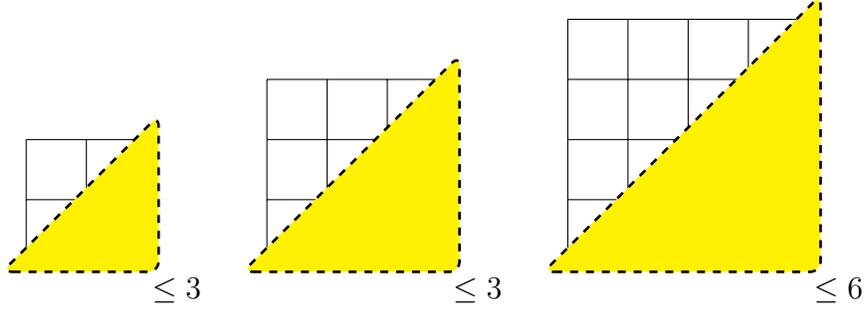
\begin{figure}[h!]
	\centering
	\begin{tikzpicture}[scale=0.8]
		\begin{scope}[xshift=0 cm,yshift=0cm]
			\draw (0,0)grid(2,2);
			\draw 
			(0,0)coordinate (a)
			(2,0)coordinate (b)
			(2,2)coordinate (c)
			;
			\xtriwindow{a}{b}{c}{3}
		\end{scope}
		
		\begin{scope}[xshift=4 cm,yshift=0cm]
			\draw (0,0)grid(3,3);
			\draw 
			(0,0)coordinate (a)
			(3,0)coordinate (b)
			(3,3)coordinate (c)
			;
			\xtriwindow{a}{b}{c}{3}
		\end{scope}
		
		\begin{scope}[xshift=9 cm,yshift=0cm]
			\draw (0,0)grid(4,4);
			\draw 
			(0,0)coordinate (a)
			(4,0)coordinate (b)
			(4,4)coordinate (c)
			;
			\xtriwindow{a}{b}{c}{6}
		\end{scope}
	\end{tikzpicture}
	\caption{\cref{claim:appendix}: The upper or lower triangular part of the grid $[k]^2$ can contain at most 3 points when $k=3,4$, and at most 6 points when $k=5$. Moreover, for $k=3$ and $k=5$, the only configurations achieving equality are shown below.}
	\label{fig:claim}
\end{figure}

\begin{figure}[h!]
	\centering
	\begin{tikzpicture}[scale=0.8]
		\begin{scope}[xshift=0 cm,yshift=0cm]
			
			\draw (0,0)grid(2,2);
			\draw 
			(0,0)coordinate (a)
			(2,0)coordinate (b)
			(2,2)coordinate (c)
			;
			\foreach \x in {a,b,c}{
				\draw (\x)node[vtx]{};
			}
			\draw (1,0)node[novtx]{};
			\draw (1,1)node[novtx]{};
			\draw (2,1)node[novtx]{};
		\end{scope}
		
		\begin{scope}[xshift=4 cm,yshift=0cm]
			
			\draw (0,0)grid(4,4);
			\draw 
			(0,0)coordinate (a)
			(4,0)coordinate (b)
			(4,4)coordinate (c)
			;
			
			\draw (0,0)node[vtx]{};
			\draw (1,1)node[vtx]{};
			\draw (2,2)node[novtx]{};
			\draw (3,3)node[vtx]{};
			\draw (4,4)node[vtx]{};
			
			\foreach \x in {0,...,2}{
				\foreach \y in {1,2}{
					\draw (\x+\y,\x)node[novtx]{};
			}}
			\draw (4,3)node[novtx]{};
			\draw (3,0)node[vtx]{};
			\draw (4,1)node[vtx]{};
			\draw (4,0)node[novtx]{};
			
		\end{scope}
		
		\begin{scope}[xshift=10 cm,yshift=0cm]
			
			\draw (0,0)grid(4,4);
			\draw 
			(0,0)coordinate (a)
			(4,0)coordinate (b)
			(4,4)coordinate (c)
			;
			
			\draw (0,0)node[vtx]{};
			\draw (1,1)node[novtx]{};
			\draw (2,2)node[vtx]{};
			\draw (3,3)node[novtx]{};
			\draw (4,4)node[vtx]{};
			
			\foreach \x in {0,...,3}{
				\draw (\x+1,\x)node[novtx]{};
			}
			
			\draw 
			(2,0)coordinate (a)
			(4,0)coordinate (b)
			(4,2)coordinate (c)
			;
			\foreach \x in {a,b,c}{
				\draw (\x)node[vtx]{};
			}
			\draw (3,0)node[novtx]{};
			\draw (3,1)node[novtx]{};
			\draw (4,1)node[novtx]{};
			
		\end{scope}
		
	\end{tikzpicture}
	\caption{The ``optimal'' equality cases in \cref{claim:appendix}.}
	\label{fig:equality_claim}
\end{figure}
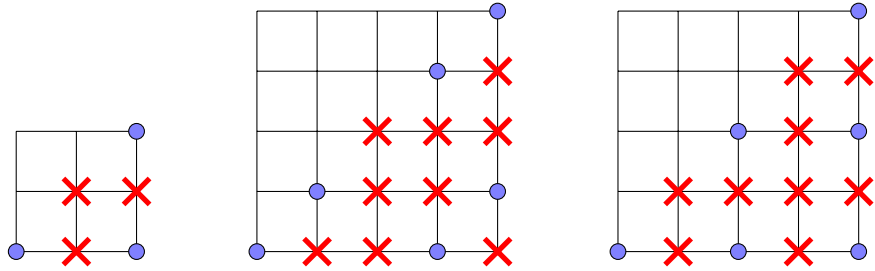

\begin{proof}
	In each of the cases, let $[k]_L=\{(x,y)\in\mathbb{Z}^2: 1\leq y\leq x\leq k\}$, and $T=S\cap [k]_L$. The proofs for the upper triangular part are immediate by symmetry.
	
	For $k=3$, note that $T\cap [2]_L$ either contains at most one point or is $\{(1,1),(2,2)\}$. Now condition on the number $t_3$ of points in the last column, which is at most 2 because of the distance 1 restriction. If $t_3=0$, then clearly $|T|\leq 2$. If $t_3=1$, then regardless of which point of the last column is in $T$, it forms three points of a border triangle with $(1,1)$ and $(2,2)$, and hence $|T|\le 2$. Finally,
	if $t_3=2$, then the second column must be empty and so either the $(1,1)$ is not in $T$ and $|T|=2$, or it is, in which case we have the first $k=3$ case in \cref{fig:equality_claim}.
	
	For $k=4$, we break $[4]_L$ into the last column and $[3]_L$. We condition on the number $t_4$ of points of $T$ in the last column, which is at most 2 because of the distance 1 restriction. If $t_4=0$ then all points of $T$ are in $[3]_L$, so as before $|T|\leq 3$. If $t_4=1$ then there can be at most one point of $T$ in the third column, and the extra non-diagonal assumption in this $k=4$ case implies that there is at most one point of $T$ in $[2]_L$, for a total of $|T|\leq 3$. Finally, if $t_4=2$ then we conclude as in \cref{fig:claim4}:
	
	\begin{figure}[h!]
		\centering
		\begin{tikzpicture}[scale=0.8]
			\begin{scope}[xshift=0 cm,yshift=0cm]
				
				\draw (0,0)grid(3,3);
				\draw
				(3,0)coordinate (a)
				(3,1)coordinate (b)
				(3,2)coordinate (c)
				(3,3)coordinate (d)
				(2,0)coordinate (e)
				(2,1)coordinate (f)
				(2,2)coordinate (g)
				;
				\draw (a)node[vtx]{};
				\draw (b)node[novtx]{};
				\draw (c)node[novtx]{};
				\draw (d)node[vtx]{};
				\draw (e)node[novtx]{};
				\draw (f)node[novtx]{};
				\draw (g)node[novtx]{};
				
				\draw
				(0,0)coordinate (h)
				(1,0)coordinate (i)
				(1,1)coordinate (j)
				;
				\xtriwindow{h}{i}{j}{1};
			\end{scope}
			
			\begin{scope}[xshift=7 cm,yshift=0cm]
				
				\draw (0,0)grid(3,3);
				\draw
				(3,0)coordinate (a)
				(3,1)coordinate (b)
				(3,2)coordinate (c)
				(3,3)coordinate (d)
				(2,0)coordinate (e)
				(2,1)coordinate (f)
				(2,2)coordinate (g)
				;
				\draw (a)node[vtx]{};
				\draw (b)node[novtx]{};
				\draw (c)node[vtx]{};
				\draw (d)node[novtx]{};
				\draw (e)node[novtx]{};
				\draw (f)node[novtx]{};
				\draw (g)node[novtx]{};
				
				\draw
				(0,0)coordinate (h)
				(1,0)coordinate (i)
				(1,1)coordinate (j)
				;
				\xtriwindow{h}{i}{j}{1};
			\end{scope}
			
			\begin{scope}[xshift=14 cm,yshift=0cm]
				
				\draw (0,0)grid(3,3);
				\draw
				(3,0)coordinate (a)
				(3,1)coordinate (b)
				(3,2)coordinate (c)
				(3,3)coordinate (d)
				(2,0)coordinate (e)
				(2,1)coordinate (f)
				(2,2)coordinate (g)
				;
				\draw (a)node[novtx]{};
				\draw (b)node[vtx]{};
				\draw (c)node[novtx]{};
				\draw (d)node[vtx]{};
				\draw (e)node[vtx]{};
				\draw (f)node[novtx]{};
				\draw (g)node[novtx]{};
				
			\end{scope}
			
		\end{tikzpicture}
		\caption{The three possible locations of two points in the last column. In the first two cases, the entire third column is forbidden by the distance $1$ and non-diagonal assumptions, so there is only at most one other point in $T$. In the last case, the same argument holds unless $(3,1)\in T$, but then each point in $[2]_L$ is forbidden because of triangles they form with it and $(4,2)$.}
		\label{fig:claim4}
	\end{figure}
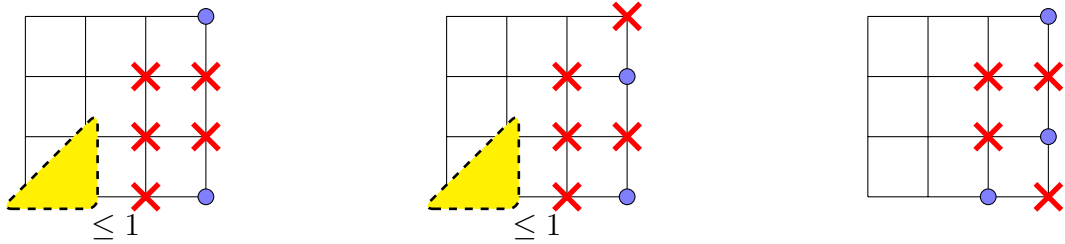
	
	For $k=5$, we cover $[5]_L$ by three pieces: the triangle $[3]_L$, the ``top triangle'' and ``the square'': 
	
	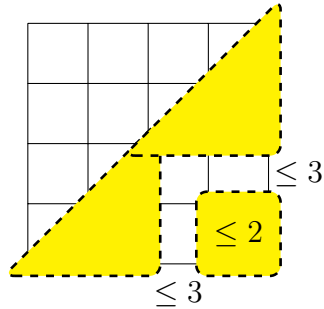
\begin{figure}[h!]
		\centering
		\begin{tikzpicture}[scale=0.8]
			\begin{scope}[xshift=9cm,yshift=0cm]
				\draw (0,0)grid(4,4);
				\draw 
				(0,0)coordinate (a)
				(2,0)coordinate (b)
				(2,2)coordinate (c)
				(4,2)coordinate (d)
				(4,4)coordinate (e)
				(3,1)coordinate (f)
				(4,0)coordinate (g)
				;
				\xtriwindow{a}{b}{c}{3}
				\xtriwindow{c}{d}{e}{3}
				\xcenteredwindow{f}{g}{2}
			\end{scope}
		\end{tikzpicture}
		\caption{Covering $[5]_L$ by two triangles and a square. Note that both triangles have at most 3 points by the $k=3$ case above, and the point $(3,3)$ is in both triangles.}
	\end{figure}
	
	We condition on the number of points $t_s$ in the square, which is at most 2 because of the distance 1 condition. The cases $t_s=0$ and $t_s=1$ can be handled simultaneously. In both, $|T|\leq 5$ unless one triangle has 3 points of $T$ and the other has at least 2 points of $T$ distinct from the first triangle. Without loss of generality, suppose $T\cap [3]_L$ has 3 points; this implies it must have the optimal $k=3$ configuration, and hence $(3,3)\in T$. But then the top triangle must have two additional points of $T$, that is, three points of $T$ as well, yielding the same configuration. One then easily checks that the only point of the square which can be in $T$ is $(5,1)$, and including it yields the second $k=5$ case in \cref{fig:equality_claim}.
	
	\begin{figure}[h!]
		\centering
		\begin{tikzpicture}[scale=0.8]
			\begin{scope}[xshift=0 cm,yshift=0cm]
				
				\draw (0,0)grid(4,4);
				\draw
				(0,0)coordinate (a)
				(2,0)coordinate (b)
				(2,2)coordinate (c)
				(4,2)coordinate (d)
				(4,4)coordinate (e)
				;
				\draw (a)node[vtx]{};
				\draw (b)node[vtx]{};
				\draw (c)node[vtx]{};
				\draw (d)node[vtx]{};
				\draw (e)node[vtx]{};
				
				\draw
				(3,0)coordinate (f)
				(3,1)coordinate (g)
				(4,1)coordinate (h)
				;
				\draw (f)node[novtx]{};
				\draw (g)node[novtx]{};
				\draw (h)node[novtx]{};
				
			\end{scope}
			
			\begin{scope}[xshift=7 cm,yshift=0cm]
				
				\draw (0,0)grid(4,4);
				\foreach \i in {0,1,2}{
					\draw (2,\i)node[novtx]{};
					\draw (2+\i,2)node[novtx]{};
				}
				\foreach \x in {0,1,3,4}{
					\draw (\x,\x)node[vtx]{};
				}
				\draw (3,3)node[right, blue]{$1$};
				\draw (4,4)node[right, blue]{$2$};
				\draw (3,1)node[red]{$1,2$};
			\end{scope}
			
		\end{tikzpicture}
		\caption{When $t_s<2$, then there are enough points in $T$ only if both triangles must have the optimal $k=3$ configuration. When $t_s=2$, the third row and columns have no points of $T$. In either case, one of the optimal $k=5$ configurations is achieved.}
		\label{fig:claim5}
	\end{figure}
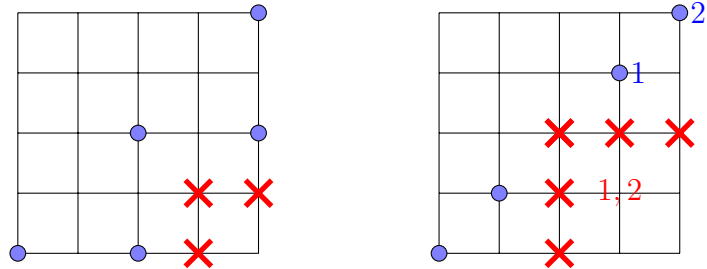
	
	Finally, suppose that $t_s=2$. Any pair of two points of the square being in $T$ forbids the four points adjacent to the square. In particular, this forbids the optimal $k=3$ configuration in either triangle, and so both have at most 2 points of $T$. Thus, $|T|\leq 5$ unless both have exactly 2 points of $T$, neither of which is $(3,3)$. This implies that the third column has no points in $T$, which means that $T\cap [2]_L=\{(1,1),(2,2)\}$ as previously noted. Analogously, the two points of $T$ in the upper triangle are forced to be $(4,4)$ and $(5,5)$. Finally, these two points form a unimodular triangle with $(4,2)$, which fixes the pair of points of $T$ in the square, and in particular yields the first $k=5$ case in \cref{fig:equality_claim}.
\end{proof}

This claim for $k=3$ and $k=5$, together with its equality cases (\cref{fig:equality_claim}) deal with the cases when there are points at distance $\sqrt{2}$ on the main diagonal or a second diagonal (see \cref{fig:maindiagonal,fig:2maindiagonal}).

\begin{figure}[h!]
	\centering
	\begin{tikzpicture}[scale=0.8]
		\begin{scope}[xshift=0 cm,yshift=0cm]
			
			\draw (0,0)grid(5,5);
			
			\foreach \x in {0,1,2,3}{
				\foreach \y in {1,2}{
					\draw (\x,\y+\x)node[novtx]{};
			}}
			
			\foreach \x in {2,3,4,5}{
				\foreach \y in {1,2}{
					\draw (\x,-\y+\x)node[novtx]{};
			}}
			
			\draw (4,5)node[novtx]{};
			\draw (1,0)node[novtx]{};
			
			\draw
			(3,0)coordinate (a1)
			(5,0)coordinate (a2)
			(5,2)coordinate (a3)
			;
			\xtriwindow{a1}{a2}{a3}{3}
			
			\draw
			(0,5)coordinate (b1)
			(2,5)coordinate (b2)
			(0,3)coordinate (b3)
			;
			\xinvtriwindow{b1}{b2}{b3}{3}
			
		\end{scope}
		
		\begin{scope}[xshift=7 cm,yshift=0cm]
			
			\draw (0,0)grid(5,5);
			
			\foreach \x in {0,1,2,3}{
				\foreach \y in {1,2}{
					\draw (\x,\y+\x)node[novtx]{};
			}}
			
			\foreach \x in {2,3,4,5}{
				\foreach \y in {1,2}{
					\draw (\x,-\y+\x)node[novtx]{};
			}}
			
			\draw (4,5)node[novtx]{};
			\draw (1,0)node[novtx]{};
			
			\draw
			(3,0)coordinate (a1)
			(5,0)coordinate (a2)
			(5,2)coordinate (a3)
			;
			
			\draw
			(0,5)coordinate (b1)
			(2,5)coordinate (b2)
			(0,3)coordinate (b3)
			;
			
			\foreach \x in {a1,a2,a3, b1, b2, b3}{
				\draw (\x)node[vtx]{};
			}
			
			\draw (0,4)node[novtx]{};
			\draw (1,4)node[novtx]{};
			\draw (1,5)node[novtx]{};
			\draw (4,0)node[novtx]{};
			\draw (4,1)node[novtx]{};
			\draw (5,1)node[novtx]{};
			
		\end{scope}
		
		\begin{scope}[xshift=14 cm,yshift=0cm]
			
			\draw (0,0)grid(5,5);
			
			\foreach \x in {0,1,2,3}{
				\foreach \y in {1,2}{
					\draw (\x,\y+\x)node[novtx]{};
			}}
			
			\foreach \x in {2,3,4,5}{
				\foreach \y in {1,2}{
					\draw (\x,-\y+\x)node[novtx]{};
			}}
			
			\draw (4,5)node[novtx]{};
			\draw (1,0)node[novtx]{};
			
			\draw
			(3,0)coordinate (a1)
			(5,0)coordinate (a2)
			(5,2)coordinate (a3)
			;
			
			\draw
			(0,5)coordinate (b1)
			(2,5)coordinate (b2)
			(0,3)coordinate (b3)
			;
			
			\draw (0,4)node[novtx]{};
			\draw (1,4)node[novtx]{};
			\draw (1,5)node[novtx]{};
			\draw (4,0)node[novtx]{};
			\draw (4,1)node[novtx]{};
			\draw (5,1)node[novtx]{};
			
			\foreach \x in {a1,a2,a3, b1, b2, b3}{
				\draw (\x)node[vtx]{};
			}
			\foreach \x in {1,2,3}{
				\draw (b\x)node[right, blue]{$\x$};
			}
			
			\draw (4,4)node[red]{$1,2$};
			\draw (1,1)node[red]{$1,3$};
			
			\draw (a1)node[right, blue]{$5$};
			\draw (a2)node[right, blue]{$6$};
			\draw (a3)node[right, blue]{$4$};
			
			\draw (2,2)node[red]{$1,5$};
			\draw (3,3)node[red]{$2,6$};
			
		\end{scope}
		
	\end{tikzpicture}
	\caption{There are two points in $T$ at distance $\sqrt{2}$ on the main diagonal. We apply the equality case for $k=3$ in \cref{claim:appendix} to conclude $T$ has at most 9 points.}
	\label{fig:maindiagonal}
\end{figure}
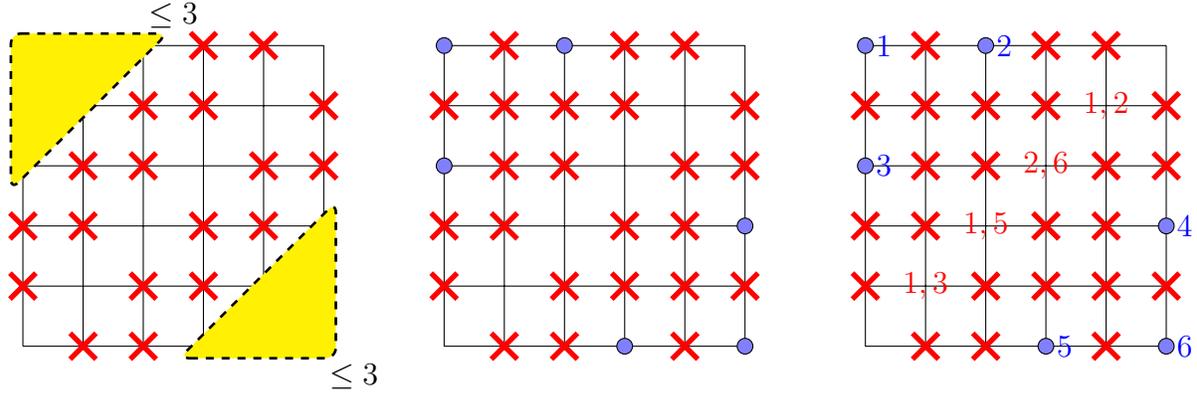

\begin{figure}[h!]
	\centering
	\begin{tikzpicture}[scale=0.8]
		\begin{scope}[xshift=0 cm,yshift=0cm]
			
			\draw (0,0)grid(5,5);
			\draw
			(0,2)coordinate (a)
			(1,3)coordinate (b)
			;
			\draw (a)node[vtx]{};
			\draw (b)node[vtx]{};
			
			\foreach \x in {0,1,2,3,4}{
				\foreach \y in {0,1}{
					\draw (\x,\y+\x)node[novtx]{};
			}}
			
			\foreach \x in {0,1}{
				\foreach \y in {3,4}{
					\draw (\x,\y+\x)node[novtx]{};
			}}
			
			\draw (2,5)node[novtx]{};
			\draw (2,4)node[novtx]{};
			\draw (5,5)node[novtx]{};
			
			\draw
			(1,0)coordinate (a1)
			(5,0)coordinate (a2)
			(5,4)coordinate (a3)
			;
			\xtriwindow{a1}{a2}{a3}{6}
			
		\end{scope}
		
		\begin{scope}[xshift=7 cm,yshift=0cm]
			
			\draw (0,0)grid(5,5);
			\draw
			(0,2)coordinate (a)
			(1,3)coordinate (b)
			;
			\draw (a)node[vtx]{};
			\draw (b)node[vtx]{};
			
			\foreach \x in {0,1,2,3,4}{
				\foreach \y in {0,1}{
					\draw (\x,\y+\x)node[novtx]{};
			}}
			
			\foreach \x in {0,1}{
				\foreach \y in {3,4}{
					\draw (\x,\y+\x)node[novtx]{};
			}}
			
			\draw (2,5)node[novtx]{};
			\draw (2,4)node[novtx]{};
			\draw (5,5)node[novtx]{};

			\draw (0,5)node[vtx]{};
			\draw (0,5)node[right, blue]{$1$};
			
			\draw (3,5)node[vtx]{};
			
			\draw (1,3)node[right, blue]{$2$};
			
			\draw (2,1)node[red]{$1,2$};
			\draw (3,0)node[red]{$1,2$};
		\end{scope}
		
	\end{tikzpicture}
	\caption{There are two points in $T$ at distance $\sqrt{2}$ on a second diagonal. We apply the equality case for $k=5$ in \cref{claim:appendix} to conclude $T$ has at most 9 points.}
	\label{fig:2maindiagonal}
\end{figure}

Now, we assume there are no points at distance $\sqrt{2}$ either on the main diagonal or a second diagonal. The latter assumption allows us to use of the claim for $k=4$ when there are points at distance $\sqrt{2}$ on a first diagonal (see \cref{fig:1maindiagonal}) or a third diagonal (see \cref{fig:3maindiagonal}). 

\begin{figure}[h!]
	\centering
	\begin{tikzpicture}[scale=0.8]
		\begin{scope}[xshift=0 cm,yshift=0cm]
			
			\draw (0,0)grid(5,5);
			\draw
			(0,1)coordinate (a)
			(1,2)coordinate (b)
			;
			\draw (a)node[vtx]{};
			\draw (b)node[vtx]{};
			
			\foreach \x in {0,1,2}{
				\foreach \y in {2,3}{
					\draw (\x,\y+\x)node[novtx]{};
			}}
			
			\foreach \x in {1,2,3,4,5}{
				\foreach \y in {0,1}{
					\draw (\x,-\y+\x)node[novtx]{};
			}}
			
			\draw (2,3)node[novtx]{};
			\draw (0,0)node[novtx]{};
			\draw (3,5)node[novtx]{};
			
			\draw
			(2,0)coordinate (a1)
			(5,0)coordinate (a2)
			(5,3)coordinate (a3)
			;
			\xtriwindow{a1}{a2}{a3}{3}
			
			\draw
			(0,5)coordinate (b1)
			(1,5)coordinate (b2)
			(0,4)coordinate (b3)
			;
			\xinvtriwindow{b1}{b2}{b3}{2}
			
		\end{scope}
		
	\end{tikzpicture}
	\caption{There are two points in $T$ at distance $\sqrt{2}$ on a first diagonal. We apply the $k=4$ case of \cref{claim:appendix} to conclude $T$ has at most 9 points.}
	\label{fig:1maindiagonal}
\end{figure}
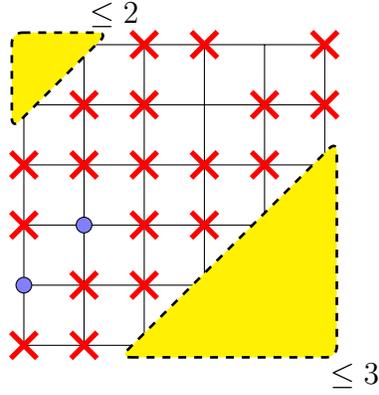

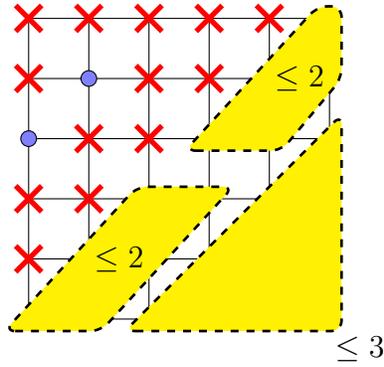
\begin{figure}[h!]
	\centering
	\begin{tikzpicture}[scale=0.8]
		\begin{scope}[xshift=0 cm,yshift=0cm]
			
			\draw (0,0)grid(5,5);
			\draw
			(0,3)coordinate (a)
			(1,4)coordinate (b)
			;
			\draw (a)node[vtx]{};
			\draw (b)node[vtx]{};
			
			\foreach \x in {0,1,2,3}{
				\foreach \y in {1,2}{
					\draw (\x,\y+\x)node[novtx]{};
			}}
			
			\draw (0,5)node[novtx]{};
			\draw (0,4)node[novtx]{};
			\draw (1,5)node[novtx]{};
			\draw (2,5)node[novtx]{};
			\draw (4,5)node[novtx]{};
		\end{scope}
		
		\begin{scope}[xshift=0 cm,yshift=0cm]
			\draw
			(0,-1)coordinate (a)
			(0,0)coordinate (b1)
			(1,0)coordinate (b2)
			(3,2)coordinate (b3)
			(2,2)coordinate (b4)
			(3,3)coordinate (b5)
			(4,3)coordinate (b6)
			(5,4)coordinate (b7)
			(5,5)coordinate (b8)
			(2,0)coordinate (a1)
			(5,0)coordinate (a2)
			(5,3)coordinate (a3)
			;
			\xskewwindow{b1}{b2}{b3}{b4}
			\xfivewindow{b5}{b6}{b7}{b8}{b8}
			\xtriwindow{a1}{a2}{a3}{3}
		\end{scope}
		
	\end{tikzpicture}
	\caption{There are two points in $T$ at distance $\sqrt{2}$ on a third diagonal. We make use of skewed $2 \times 3$ windows and the $k=4$ case of \cref{claim:appendix} to conclude $T$ has at most 9 points.}
	\label{fig:3maindiagonal}
\end{figure}

We now make the following observation: the above cases show that $T$ cannot have more than $9$ points unless the only points at distance $\sqrt{2}$ are those of either a fourth diagonal or fourth antidiagonal. Without loss of generality, let $T$ contain the two points of the upper fourth diagonal.

Suppose also that $T$ \emph{does not} contain the two points of the lower fourth diagonal. We partition the grid as in \cref{fig:lastcase2} and observe that if $|T|>9$ then each of the six windows must contain the maximum number of points. But there are only two ways to include $2$ points in the `left $2\times 3$ window' avoiding distance $1$ or $\sqrt{2}$ on a lower diagonal; neither yields enough points.

\begin{figure}[h!]
	\centering
	\begin{tikzpicture}[scale=0.8]
		\begin{scope}[xshift=0 cm,yshift=0cm]
			\draw (0,0)grid(5,5);
			\draw
			(0,4)coordinate (a)
			(1,5)coordinate (b)
			;
			\draw (a)node[vtx]{};
			\draw (b)node[vtx]{};
			
			\foreach \x in {0,1,2}{
				\foreach \y in {2,3}{
					\draw (\x,\y+\x)node[novtx]{};
			}}
			
			\draw (0,5)node[novtx]{};
			\draw (3,5)node[novtx]{};
		\end{scope}
		
		\foreach \x/\y in {2/0, 2/2, 4/2}{
			\begin{scope}[xshift=\x cm,yshift=\y cm]
				\draw
				(0,1)coordinate (a)
				(1,0)coordinate (b)
				;
				\xcenteredwindow{a}{b}{1}
			\end{scope}
		}
		
		\begin{scope}[xshift=4 cm,yshift=0 cm]
			\draw
			(0,1)coordinate (a)
			(1,0)coordinate (b)
			;
			\xcenteredwindow{a}{b}{1}
		\end{scope}
		
		\begin{scope}[xshift=0 cm,yshift=0 cm]
			\draw
			(0,2)coordinate (a)
			(1,0)coordinate (b)
			;
			\xcenteredwindow{a}{b}{2}
		\end{scope}
		
		\begin{scope}[xshift=3 cm,yshift=4 cm]
			\draw
			(0,1)coordinate (a)
			(2,0)coordinate (b)
			;
			\xcenteredwindow{a}{b}{2}
		\end{scope}

		\begin{scope}[xshift=7 cm,yshift=0cm]
			\draw (0,0)grid(5,5);
			\draw
			(0,4)coordinate (a)
			(1,5)coordinate (b)
			;
			\draw (a)node[vtx]{};
			\draw (b)node[vtx]{};
			
			\foreach \x in {0,1,2}{
				\foreach \y in {2,3}{
					\draw (\x,\y+\x)node[novtx]{};
			}}
			
			\draw (0,5)node[novtx]{};
			\draw (3,5)node[novtx]{};
			
			\draw (0,1)node[vtx]{};
			\draw (1,0)node[vtx]{};
			
			\draw (0,0)node[novtx]{};
			\draw (1,1)node[novtx]{};
			\draw (1,2)node[novtx]{};
			
			\draw (0,1)node[right, blue]{$1$};
			\draw (1,0)node[right, blue]{$2$};
			
			\draw (2,0)node[red]{$1,2$};
			\draw (2,1)node[red]{$1,2$};
			\draw (3,0)node[red]{$1,2$};
			
			\draw (3,1)node[vtx]{};
			\draw (3,1)node[right, blue]{$3$};
			
			\draw (4,2)node[red]{$2,3$};
			\draw (5,2)node[red]{$2,3$};
			\draw (5,3)node[red]{$2,3$};
			
			\draw (4,3)node[vtx]{};
			\draw (4,3)node[right, blue]{$4$};
			
			\draw (5,4)node[red]{$3,4$};
			\draw (5,5)node[red]{$3,4$};
			\draw (4,5)node[red]{$3,4$};
			\draw (4,4)node[red]{$3,4$};
			
		\end{scope}
		
		\begin{scope}[xshift=14 cm,yshift=0cm]
			\draw (0,0)grid(5,5);
			\draw
			(0,4)coordinate (a)
			(1,5)coordinate (b)
			;
			\draw (a)node[vtx]{};
			\draw (b)node[vtx]{};
			
			\foreach \x in {0,1,2}{
				\foreach \y in {2,3}{
					\draw (\x,\y+\x)node[novtx]{};
			}}
			
			\draw (0,5)node[novtx]{};
			\draw (3,5)node[novtx]{};
			
			\draw (0,0)node[vtx]{};
			\draw (0,0)node[right, blue]{$1$};
			
			\draw (1,2)node[vtx]{};
			\draw (1,2)node[right, blue]{$2$};
			
			\draw (0,1)node[novtx]{};
			\draw (1,0)node[novtx]{};
			\draw (1,1)node[novtx]{};
			
			\draw (3,4)node[red]{$1,2$};
			
			\draw (4,5)node[vtx]{};
			\draw (4,5)node[right, blue]{$3$};
			\draw (5,4)node[vtx]{};
			\draw (5,4)node[right, blue]{$4$};
			
		\end{scope}
		
	\end{tikzpicture}
	\caption{In the first case, we conclude there cannot be 2 points in the top $3\times 2$ window. The second case forces $3$ and $4$ to be in $T$ and so it is analogous to the first case by symmetry.
		\label{fig:lastcase2}}
\end{figure}
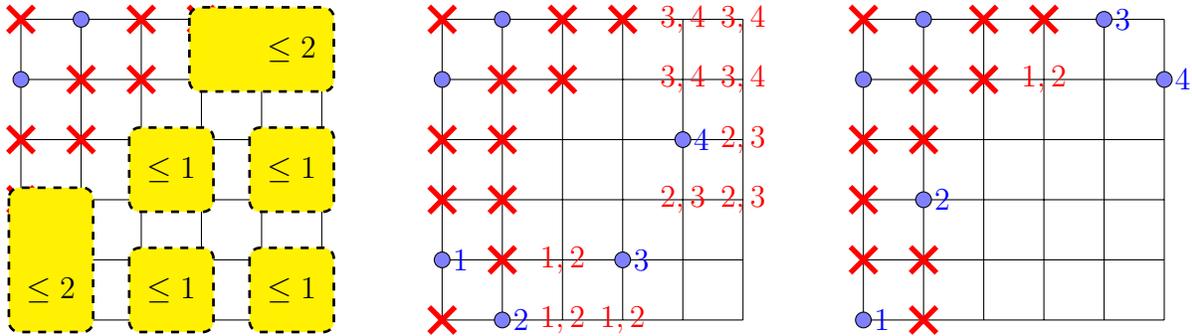

Thus, the last remaining case is when $T$ contains both fourth diagonals, in which case $|T|>9$ can be achieved only if at most one point missing; that is, one of the five remaining windows can have one fewer than its maximum number of points.

Using the notation from \cref{fig:lastcase1}, at least one of points 1 and 2 are in $T$; by symmetry, we may say that point 1 is in $T$. Then the top $3\times 2$ window contains at most one point, and $|T|>9$ only if both 2 and 4 are in $T$. 
If 2 is in $T$, then there is at most one point in both rectangular corner windows. We conclude $|T| \le 9$.

\begin{figure}[h!]
	\centering
	\begin{tikzpicture}[scale=0.8]
		\begin{scope}[xshift=0 cm,yshift=0cm]
			\draw (0,0)grid(5,5);
			\draw
			(0,4)coordinate (a)
			(1,5)coordinate (b)
			;
			\draw (a)node[vtx]{};
			\draw (b)node[vtx]{};
			
			\foreach \x in {0,1,2}{
				\foreach \y in {2,3}{
					\draw (\x,\y+\x)node[novtx]{};
			}}
			
			\draw (0,5)node[novtx]{};
			\draw (3,5)node[novtx]{};
		\end{scope}
		
		\foreach \x/\y in {2/0, 2/2, 4/2}{
			\begin{scope}[xshift=\x cm,yshift=\y cm]
				\draw
				(0,1)coordinate (a)
				(1,0)coordinate (b)
				;
				\xcenteredwindow{a}{b}{1}
			\end{scope}
		}
		
		\begin{scope}[xshift=4 cm,yshift=0 cm]
			\draw
			(0,1)coordinate (a)
			(1,0)coordinate (b)
			;
			\xcenteredwindow{a}{b}{2}
		\end{scope}
		
		\begin{scope}[xshift=0 cm,yshift=0 cm]
			\draw
			(0,2)coordinate (a)
			(1,0)coordinate (b)
			;
			\xcenteredwindow{a}{b}{2}
		\end{scope}
		
		\begin{scope}[xshift=3 cm,yshift=4 cm]
			\draw
			(0,1)coordinate (a)
			(2,0)coordinate (b)
			;
			\xcenteredwindow{a}{b}{2}
		\end{scope}
		
		\begin{scope}[xshift=7 cm,yshift=0cm]
			\draw (0,0)grid(5,5);
			\draw
			(0,4)coordinate (a)
			(1,5)coordinate (b)
			;
			\draw (a)node[vtx]{};
			\draw (b)node[vtx]{};
			
			\foreach \x in {0,1,2}{
				\foreach \y in {2,3}{
					\draw (\x,\y+\x)node[novtx]{};
			}}
			
			\draw (0,5)node[novtx]{};
			\draw (3,5)node[novtx]{};
			
			\draw (4,0)node[vtx]{};
			\draw (5,1)node[vtx]{};
			
			\draw (5,0)node[novtx]{};
			\draw (2,0)node[novtx]{};
			\draw (3,1)node[novtx]{};
			\draw (4,2)node[novtx]{};
			\draw (5,3)node[novtx]{};
			\draw (3,0)node[novtx]{};
			\draw (4,1)node[novtx]{};
			\draw (5,2)node[novtx]{};
			
			\draw (4,3)node[vtx]{};
			\draw (2,1)node[vtx]{};
			\draw (4,3)node[right, blue]{$1$};
			\draw (2,1)node[right, blue]{$2$};
			
			\draw (5,1)node[right, blue]{$3$};
			\draw (4,0)node[right, blue]{$5$};
			
			\draw (3,4)node[red]{$1,3$};
			\draw (4,4)node[red]{$1,3$};
			\draw (4,5)node[red]{$1,3$};
			\draw (5,4)node[red]{$1$};
			
			\draw (5,5)node[right, blue]{$4$};
			\draw (5,5)node[vtx]{};
			
			\draw (0,4)node[vtx]{};
			
			\draw (0,1)node[red]{$2,5$};
			\draw (1,0)node[red]{$2,5$};
			\draw (1,1)node[red]{$2,5$};
			\draw (1,2)node[red]{$2,5$};
		\end{scope}
		
	\end{tikzpicture}
	\caption{Without loss of generality, point $1$ is in $T$. But then both the points 2 and 4 must also be in $T$, and $|T|\leq 9$.}
	\label{fig:lastcase1}
\end{figure}
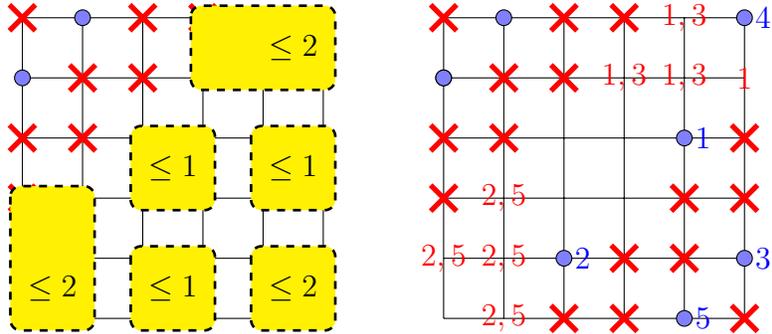

We thus conclude that $|T|\leq 9$ if it has any consecutive points in any diagonal of $[6]$, as desired.
	
\end{document}